\documentclass{amsart}
\usepackage[utf8]{inputenc}
\usepackage{amssymb,amsthm}
\usepackage{amsmath}
\usepackage{amscd}
\usepackage{tikz}
\usepackage{mathtools,tikz-cd}
\usepackage{mathabx}
\usepackage[all]{xy}
\usepackage{hyperref}

\title{K-stability of log del Pezzo hypersurfaces with index 2}
\author{In-Kyun Kim, Nivedita Viswanathan, Joonyeong Won}
\address{Department of Mathematics, Yonsei University, Seoul, Korea}
\email{soulcraw@gmail.com}

\address{Department of Mathematical Sciences, Loughborough University}
\email{N.Viswanathan@lboro.ac.uk}

\address{Department of Mathematics, Ewha Womans University, Seoul, Korea}
\email{leonwon@kias.re.kr}

\newtheorem{theorem}{Theorem}[subsection]

\newtheorem{definition}[theorem]{Definition}
\newtheorem{corollary}[theorem]{Corollary}
\newtheorem{proposition}[theorem]{Proposition}
\newtheorem{lemma}[theorem]{Lemma}
\newtheorem*{Ack}{Acknowledgments}

\theoremstyle{remark}
\newtheorem*{remark}{Remark}

\newcommand{\Supp}{\operatorname{Supp}}
\newcommand{\Sing}{\operatorname{Sing}}

\newcommand{\mult}{\operatorname{mult}}

\newcommand{\wt}{\operatorname{wt}}
\newcommand{\vol}{\operatorname{vol}}
\newcommand{\dd}{\mathop{}\!\mathrm{d}}

\newcommand{\PP}{{\mathbb P}}
\newcommand{\QQ}{{\mathbb Q}}
\newcommand{\RR}{{\mathbb R}}

\newcommand{\msa}{\mathsf a}
\newcommand{\msb}{\mathsf b}

\newcommand{\msp}{\mathsf p}
\newcommand{\msq}{\mathsf q}

\newcommand{\mst}{\mathsf t}

\newcommand{\mcL}{\mathcal L}

\newcommand{\mcO}{\mathcal O}

\let\oldtocsection=\tocsection
\let\oldtocsubsection=\tocsubsection

\renewcommand{\tocsection}[2]{\hspace{0em}\oldtocsection{#1}{#2}}
\renewcommand{\tocsubsection}[2]{\hspace{3em}\oldtocsubsection{#1}{#2}}

\begin{document}

\begin{abstract}
We completely classify K-stability of log del Pezzo hypersurfaces with index 2.
\end{abstract}
\maketitle

\section{Introduction}
Existence of K\"ahler-Einstein metrics on Fano varieties has been an intense field of research and the notion of K-stability has paved way to approach this question from an algebro-geometric point of view. In particular, the Yau-Tian-Donaldson Conjecture predicts that Fano varieties admits a K\"ahler-Einstein metric if and only if it is K-polystable. This is now proven in the case of Fano manifolds  \cite{CDS15,TianYTD} and in the case of singular Fano varieties (klt singularities) in \cite{Berman,SpottiSunYao,LiWangXu}.

In the cases of smooth del Pezzo surfaces $S$, Tian and Yau proved in \cite{Tian} and \cite{TianandYau} that $S$ is K-polystable if and only if it is not a blow up of $\mathbb{P}^2$ in one or two points. This solves the problem for Fano manifolds in dimension 2. But this question is widely open for Fano orbifolds.

A Fano orbifold can be embedded in a weighted projective space via the Riemann-Roch Space of the anticanonical divisor, or the smallest fraction of the divisor that exists in the class group. We are interested in the hypersurface case: quasi-smooth and well-formed hypersurfaces $S_d$ in $\mathbb{P}(a_0,a_1,a_2,a_3)$ of degree $d$, which are given by a quasihomogenoeus polynomial equation of degree $d$
$$
f(x,y,z,t)=0 \subset \mathbb{P}(a_0,a_1,a_2,a_3) \cong \mathrm{Proj}(\mathbb{C}[x,y,z,t]),
$$
where $\wt(x) = a_0$, $\wt(y) = a_1$, $\wt(z) = a_2$ and $\wt(t) = a_3$. 
The surface $S_d$ is said to be \textit{quasi-smooth} if 
$$
f(x,y,z,t)=0 \subset \mathbb{C}^4 \cong \mathrm{Spec}(\mathbb{C}[x,y,z,t]),
$$
defines a hypersurface that is singular only at the origin in $\mathbb{C}^4$ and thus implies that the surface $S_d$ can have atmost cyclic quotient singularities. It is said to be \textit{well-formed} if 
$$
\mathrm{gcd}(a_0,a_2,a_3)=\mathrm{gcd}(a_0,a_1,a_2)=\mathrm{gcd}(a_0,a_1,a_3)=\mathrm{gcd}(a_1,a_2,a_3)=1.
$$
Note that being well-formed implies that 
$$
-K_{S_d} \sim_{\mathbb{Q}} \mathcal{O}_{S_d}(a_0 + a_1 + a_2 + a_3 - d).
$$
Recall that the index of $S_d$ is given by $I=a_0+a_1+a_2+a_3-d$; this is indeed the divisibility of the anticanonical divisor in the class group, mentioned above. So suppose $I$ is positive, then $S_d$ is a del Pezzo surface with at most quotient singularities. In the case of quasi-smooth, well-formed hypersurfaces $S_d$ in $\mathbb{P}(a_0,a_1,a_2,a_3)$ of index $I=1$, the existence of the K\"ahler-Einstein metric was determined by Johnson and Koll\'{a}r except in few cases.
\begin{theorem}[{\cite[Theorem~8]{J&K}}]
\label{theorem:Johnson-Kollar}
Suppose that $S_d$ with $I=1$ is singular  and the quintuple $(a_{0},a_{1},a_{2},a_{3},d)$ is not one of the following four quintuples:
\begin{equation}\label{eq:four-quintuples}
(1,2,3,5,10), (1,3,5,7,15), (1,3,5,8,16),  (2,3,5,9,18).
\end{equation}
Then $S_d$ admits an orbifold K\"ahler-Einstein metric.
\end{theorem}
The primary criterion used in proving the above mentioned result was Tian's criterion for proving the existence of K\"ahler-Einstein metric, namely, if the inequality 
$$
\alpha(S_d)>\frac{2}{3}
$$ holds, then $S_d$ admits a K\"ahler-Eintein metric. This method worked for all the $I=1$ hypersurfaces, except the 4 mentioned in the theorem. Two of the cases have been treated in \cite{Ara02}.
\begin{theorem}[{\cite[Theorem~4.1]{Ara02}}]
\label{theorem:Araujo}
In the following two cases:
\begin{itemize}
\item $(a_{0},a_{1},a_{2},a_{3},d)= (1,2,3,5,10)$,
\item $(a_{0},a_{1},a_{2},a_{3},d)=(1,3,5,7,15)$ and the equation of $S_d$ contains~$yzt$,
\end{itemize}
the inequality $\alpha(S_d)>\frac{2}{3}$ holds. In particular, $S_d$ admits an orbifold K\"ahler-Einstein metric.
\end{theorem}
The two other remaining cases have been shown in \cite{CheltsovParkShramovJGA}.

\begin{theorem}[{\cite[Theorem~1.10]{CheltsovParkShramovJGA}}]
\label{theorem:old-I-1}
Suppose that the quintuple $(a_{0},a_{1},a_{2},a_{3},d)$ is either $(1,3,5,8,16)$ or $(2,3,5,9,18)$. Then $\alpha(S_d)>\frac{2}{3}$. In particular, $S_d$ admits an orbifold K\"ahler-Einstein metric.
\end{theorem}
The only remaining case was when $(a_{0},a_{1},a_{2},a_{3},d)=(1,3,5,7,15)$ and the equation of $S_d$ does not contain~$yzt$. But this was solved in \cite{delta-inv-of-sing-dP}. In this case, since $\alpha(S_d)<\frac{2}{3}$, the $\alpha$-invariant criterion could not be used. Instead, an invariant called the $\delta$-invariant, that was introduced by Fujita and Odaka in \cite{FujitaOdaka}, has been used to establish the existence of K\"ahler-Einstein metrics. This invariant serves as a strong criterion to establish uniform K-stability.

\begin{theorem}
\cite[Theorem 1.7]{delta-inv-of-sing-dP}
Let $S_d$ be a quasi-smooth hypersurface in $\mathbb{P}(1,3,5,7)$ of degree $15$ such that its
defining equation does not contain~$yzt$.
Then $\delta(S_d)\geqslant\frac{6}{5}$.
In particular, the surface $S_d$ admits an orbifold K\"ahler-Einstein metric.
\end{theorem}

Therefore, this completely solves the problem of existence of K\"ahler-Einstein metrics for $I=1$.  

Next, when the index $I$ of $S_d$ is 2, the existence of K\"{a}hler-Einstein metrics on $S_d$ has been studied in \cite{BGN03}, \cite{CheltsovParkShramovJGA}, \cite{dPzoo}, using $\alpha$-invariant computations. Note that this approach works under the assumption that $I < 3a_0/2$ and $(a_0,a_1,a_2,a_3,d) \neq (I-k,I+k,a,a+k,2a+k+I)$, for any non-negative integer $k<I$ and any positive integer $a \geq I+k$, because if $I \geq 3a_0/2$ or if  $I < 3a_0/2$ and $(a_0,a_1,a_2,a_3,d)= (I-k,I+k,a,a+k,2a+k+I)$, then $\alpha(S_d)\leq 2/3$ (\cite[Corollary 1.15]{dPzoo}) where $\alpha(S_d)$ is the $\alpha$-invariant of $S_d$ and hence cannot be used to prove the existence of K\"{a}hler-Einstein metrics. Recently, the existence of K\"ahler-Einstein metrics on few of the remaining cases has been shown using the $\delta$-invariant in \cite{delta-inv-of-sing-dP}. 

As a consequence, it was conjectured in \cite[Conjecture 1.10]{delta-inv-of-sing-dP}, that for $I=2$, all $S_d$ admit an orbifold K\"ahler-Einstein metric. But this was disproved by Kim and Won in \cite[Theorem 1.2]{KW21}.

\begin{theorem}\cite[Theorem 1.2]{KW21}
    Suppose that $S_d$ is a quasi-smooth hypersurface such that the quintuple $(a_0,a_1,a_2,a_3,d)$ is one of the following:
    \begin{equation*}
        (1,6,9,13,27),\quad (1,9,15,22,45), \quad (1,3,3n+3,3n+4,6n+9),
    \end{equation*}
    \begin{equation*}
        (1,1,n+1,m+1,n+m+2),\quad(1,3,3n+4,3n+5,6n+11)
    \end{equation*}
    where $n$ and $m$ are non-negative integers with $n<m$. Then $S_d$ does not have an orbifold K\"ahler-Einstein metric.
\end{theorem}
In this paper, we determine the existence of the K\"ahler-Einstein metric for the remaining quasi-smooth and well-formed hypersurfaces of index $I=2$, thus giving a complete answer in the case of Index 2.

\begin{theorem}[\textbf{Main Theorem}]
\label{main theorem}
Let $S_d$ be a quasi-smooth, well-formed hypersurface with $I=2$. The following table gives our results on the existence of K\"ahler-Einstein metrics on $S_d$ in $\PP(a_0,a_1,a_2,a_3)$ of degree $d$.
\begin{center}
    \begin{tabular}{|c|c|c|c|}
        \hline
        No. &$(a_0,a_1,a_2,a_3)$ & degree & KE\\
        \hline
        1 & $(1,1,n,n)$, $n\geq 2$ & $2n$ & yes\\
        2 & $(1,2,n+2,n+3)$, $n\geq 0$ & $2(n+3)$ & yes\\
        3 & $(1,3,4,6)$ & $12$ & yes\\
        4 & $(1,4,5,7)$ & $15$ & yes\\ 
        5 & $(1,4,5,8)$ & $16$ & yes\\
        6 & $(1,4,6,9)$ &  $18$ & yes\\
        7 & $(1, 5, 7, 11)$ & $22$ & yes\\
        8 & $(1, 6, 10, 15)$ & $30$ & yes\\
        9 & $(1, 7, 12, 18)$ & $36$ & yes\\
        10 & $(1, 8, 13, 20)$ & $40$ & yes\\
        \hline
    \end{tabular}
\end{center}
\end{theorem}

\section{Local and Global Results}

\subsection{Notations}
Throughout the paper we use the following notations:
\begin{itemize}
    \item For positive integers $a_0$, $a_1$, $a_2$ and $a_3$, $\PP(a_0,a_1,a_2,a_3)$ is the weighted projective space. We assume that $a_0\leq a_1\leq a_2\leq a_3$.
    \item We take $x$, $y$, $z$ and $t$ to be the weighted homogeneous coordinates of $\PP(a_0,a_1,a_2,a_3)$ with weights $\wt(x)=a_0$, $\wt(y)=a_1$, $\wt(z)=a_2$ and $\wt(t)=a_3$.
    \item $S_d\subset \PP(a_0,a_1,a_2,a_3)$ denotes a quasi-smooth weighted hypersurface given by a quasi-homogeneous polynomial of degree $d$.
    \item $C_x$ is the hyperplane section that is cut out by the equation $x = 0$ in $S_d$. The hyperplane sections $C_y$, $C_z$ and $C_t$ are defined in a similar way.
    \item $\msp_x$ denotes the point on $S_d$ given by $y=z=t=0$. The points $\msp_y$, $\msp_z$ and $\msp_t$ are defined in a similar way.
    \item $-K_{S_d}$ denotes the anti-canonical divisor of $S_d$.
\end{itemize}

Let $S$ be a del Pezzo surface with at most cyclic quotient singularities, and let $D$ be an effective $\QQ$-divisor on $S$. Let $C$ be an irreducible reduced curve on $S$. We can write
\begin{equation*}
    D = aC + \Delta
\end{equation*} 
where $a$ is a non-negative rational number and $\Delta$ is an effective $\QQ$-divisor such that $C$ is not contained in the support of $\Delta$. 

\subsection{Results for smooth points}
\begin{lemma}[\cite{Kol97}]\label{Foundation:mult}
    Let $\msp$ be a smooth point of $S$. Suppose that the log pair $(S, D)$ is not log canonical at the point $\msp$. Then $\mult_{\msp}(D) > 1$.
\end{lemma}

\begin{lemma}\label{Inversion of adjunction for smooth point}
    Let $\msp$ be a smooth point of $S$. Suppose that $a\leq 1$, $C$ is also smooth at $\msp$ and the log pair $(S, D)$ is not log canonical at $\msp$. Then we have the following:
    \begin{equation*}
        1< (C\cdot \Delta)_{\msp} \leq C\cdot \Delta
    \end{equation*}
    where $(C\cdot \Delta)_{\msp}$ is the local intersection number at $\msp$.
\end{lemma}

\begin{proof}
    By the assumption, the log pair $(S, C + \Delta)$ is not log canonical at $\msp$. The inversion of adjunction formula implies that the log pair $(C, \Delta|_{C})$ is not log canonical at $\msp$. Then $1 < \mult_{\msp}(\Delta|_C)$ implies the above inequality.
\end{proof}

\subsection{Results for singular points}
We will now give an analogue to the above results when the surface $S$ has cyclic quotient singularities. Suppose that $S$ has a cyclic quotient singular point $\msp$ of type $\frac{1}{n}(\msa,\msb)$ where $\msa$ and $\msb$ are coprime integers that are also coprime to $n$. 

\begin{lemma}\label{Inversion of adjunction for singular point}
    Suppose that $a\leq 1$, the log pair $(S, C)$ is purely log terminal at a point $\msp$ and the log pair $(S,C)$ is not log canonical at the point $\msp$. Then we have the following:
    \begin{equation*}
        \frac{1}{n} < C\cdot \Delta.
    \end{equation*}
\end{lemma}

\begin{proof}
    See the proof of \cite[Lemma 2.5]{CheltsovParkShramovJGA}.
\end{proof}

Recall that there is an orbifold chart $\pi\colon \widetilde{U}\to U$ for some open set $\msp\in U$ on $S$ such that $\widetilde{U}$ is smooth and $\pi$ is a cyclic cover of degree $n$ branched over $\msp$. Let $\widetilde{\msp}\in \widetilde{U}$ be a point such that $\pi(\widetilde{\msp}) = \msp$ and $\widetilde{D}=\pi^{-1}(D\vert_{U})$.

\begin{lemma}[\cite{Kol97}]\label{Foundation:mult at singular point}
The log pair $(U, D|_U)$ is log canonical at the point $\msp$ if and only if the log pair $(\widetilde{U}, \widetilde{D}|_{\widetilde{U}})$ is log canonical at the point $\widetilde{\msp}$.
\end{lemma}

For any $\QQ$-Cartier divisor $B$ in $S$ such that $B$ is not contained in the support of $D$, we write 
$$
\mult_{\msp}(D \cdot B)=\mult_{\tilde{\msp}}(\widetilde{B} \cdot \widetilde{D})
$$
 where $\widetilde{B}=\pi^{-1}(B\vert_U)$.

\begin{lemma}\label{Foundation:inequality mult intersection}
   Let $B$ be a $\QQ$-divisor in $S$ such that no component of $B$ is contained in the support of $D$. Then the inequality
    \begin{equation*}
        \frac{\mult_{\msp}(B)\mult_{\msp}(D)}{n} \leq B\cdot D
    \end{equation*}
    holds.
\end{lemma}

\begin{proof}
    It immediately follows from $B\cdot D = \sum_{\msq\in S}\frac{\mult_{\msq}(B\cdot D)}{n_{\msq}}$, where the point $\msq$ is a singular point of type $\frac{1}{n_{\msq}}(\msa_{\msq},\msb_{\msq})$.
\end{proof}

From our assumption, $S$ has a cyclic quotient singularity of type $\frac{1}{n}(\msa,\msb)$ at the point $\msp$. Let the weighted coordinates around this point $\msp$ be $x$ and $y$. Let $\phi\colon \widetilde{S}\to S$ be the weighted blow-up at $\msp$ of $S$ with weights $\wt(x) = \msa$ and $\wt(y) = \msb$. Then we have the following:
\begin{equation*}
    K_{\widetilde{S}}\equiv \phi^*(K_S) + \left(-1 + \frac{\msa}{n} + \frac{\msb}{n}\right)E
\end{equation*}
where $E$ is the exceptional divisor of $\phi$ and
\begin{equation*}
    E^2 = -\frac{n}{\msa\msb}.
\end{equation*}
Let $C$ be a curve on $S$ that is locally given by $x=0$ near $\msp$. Then we have
\begin{equation*}
    \widetilde{C}\equiv \phi^*(C)-\frac{\msa}{n}E
\end{equation*}
where $\widetilde{C}$ is the strict transform of $C$.

Let $\widetilde{D}$ be the proper transform of $D$ on $\widetilde{S}$. We have 
$$
\widetilde{D} \equiv \phi^*(D)-mE,
$$ for some non-negative rational number $m$. 

If $C$ is not contained in the support of the divisor $D$, we can bound $m$ using 
$$
0\leq \widetilde{D}\cdot\widetilde{C}=(\phi^*(D)-mE)\cdot\widetilde{C}=D\cdot C-mE\cdot\widetilde{C}.
$$ The log pullback is given by 
$$
K_{\widetilde{S}}+\widetilde{D}+\left(m-\frac{a+b-n}{n}\right)E\sim f^*(K_S+D).
$$
This implies the following.
\begin{proposition}
    The log pair $(S,D)$ is log canonical at $\msp$ if and only if the log pair 
    $$
    \left(\widetilde{S},\widetilde{D}+\left(m-\frac{\msa+\msb-n}{n}\right)E\right)
    $$ is log canonical along the curve $E$.
\end{proposition}

In the next section, we will explicitly state the main results in the theory of K-stability that are required to prove Theorem \ref{main theorem}.

\section{K-stability}\label{section:K-stability}
Let $S$ be a del Pezzo surface with Kawamata log terminal singularities and with the polarisation $-K_S$. Let $f\colon \widetilde{S}\to S$ be a~projective birational morphism with normal variety $\widetilde{S}$, and let $E$ be a~(not necessarily $f$-exceptional) prime divisor in $\widetilde{S}$.
Then $E$ is said to be a~divisor over $S$.

Let
\begin{align}\label{A_S}
A_S(E)=1+\mathrm{ord}_E\left(K_{\widetilde{S}}-f^*\left(K_S\right)\right),
\end{align}
and we let \index{$S_S(E)$}
\begin{align}\label{S_S}
    S_{-K_S}(E)=\frac{1}{(-K_S)^2}\int_{0}^{\tau}\mathrm{vol}(f^*(-K_S)-uE)\dd u,
\end{align}
where $\tau=\tau(E)$ is the~pseudo-effective threshold of $E$ with respect to $-K_S$, i.e. we have \index{pseudo-effective threshold}
$$
\tau(E)=\mathrm{sup}\left\{ u \in \mathbb{Q}_{>0}\ \middle| \ f^*(-K_S)-uE\ \text{is big}\right\}.
$$

\subsection{Stability Threshold}
The paper \cite{FujitaOdaka} introduces a~new invariant of Fano varieties, called $\delta$-invariant,
that serves as a~criterion for K-stability.
Following \cite{Blum}, we can define $\delta(S)$ as follows:
$$
\delta(S)=\delta(S;-K_S)=\inf_{E/S}\frac{A_S(E)}{S_{-K_S}(E)},
$$

where infimum is taken over all prime divisors $E$ over $S$.

In this case, the number $\delta(S)$ is also known as \emph{the stability threshold}, because of the~following result.

\begin{theorem}[{\cite{FujitaOdaka,F19,Li17,Blum,CodogniPatakfalvi,LiuXuZhuang}}]
\label{theorem:Kento-Yuji-Blum-Xu}
Let $S$ be as described above. Then
\begin{itemize}
\item $\delta(S)>1$ $\iff$ $S$ is K-stable;
\item $\delta(S)\geqslant 1$ $\iff$  $S$ is K-semistable.
\end{itemize}
\end{theorem}

We estimate $\delta(S)$ for members of families No.3 - No.10 using the method adopted in \cite{delta-inv-of-sing-dP}. Let us briefly describe this method. Here, we study log pairs $(S_d,D)$, where $S_d$ are hypersurfaces of degree $d$ in weighted projective space and $D$ is a $k$-basis type divisor for $k\gg 0$ (see \cite{FujitaOdaka} for definition) such that $D\equiv-K_{S_d}$. We then use the local analogues of the stability threshold to determine the value of $\delta_p(S)$ at various points $p \in S$, which along with the Abban-Zhuang Theory is extremely handy in computing $\delta(S)$.
We will now describe this local analogue of the stability threshold. 

\subsubsection{Local Analogues of the Stability Threshold}
For a point $\msp\in S$, we let
$$
\delta_{\msp}(S)=\inf_{\substack{E/S\\ \msp\in C_S(E)}}\frac{A_S(E)}{S_{-K_S}(E)},
$$

where infimum is taken over all prime divisors $E$ over $S$ whose centers on $S$ contain $\msp$. One can also define $\delta_{\msp}(S)$ alternatively using \textit{cool} divisors. 
\begin{definition}\cite[Definition 1.5.3]{Fano3-folds}\label{cool-divisor}
We say the divisor $D$ is cool if the inequality
\[
\mathrm{ord}_E(D) \leq S_{-K_S}(E)
\] holds for every prime Weil divisor $E$ over $S$.
\end{definition}
We then have the following proposition.
\begin{proposition}\cite[Proposition 1.5.4]{Fano3-folds}
    Let $\msp$ be a point in $S$. Then 
    \[
    \delta_{\msp}(S)=
\mathrm{sup}\left\{\lambda\in\mathbb{Q}\ \left|\ \aligned
&\text{the log pair}\ \big(S,\lambda D\big)\ \text{is log canonical at $\msp$}\\
&\text{for any effective cool $\mathbb{Q}$-divisor}\ D\sim_{\mathbb{Q}} -K_{S}\\
\endaligned\right.\right\}.
\]

\end{proposition}

Then
\begin{align*}
\delta(S)&=\inf_{\msp\in S}\delta_{\msp}(S).
\end{align*}

Note that this notion has been mentioned previously in \cite{delta-inv-of-sing-dP}, under the assumption that the divisor $D$ is a $k$-basis type divisor. Recall that $D$ is an effective $\mathbb{Q}$ divisor such that $D = aC + \Delta$ where $C$ is an irreducible, reduced curve, $a$ is a non-negative rational number and $\Delta$ is an effective $\QQ$-divisor such that $C$ is not contained in the support of $\Delta$. 

\begin{theorem}[{\cite[Theorem 2.9]{delta-inv-of-sing-dP}}]
    \label{Foundation:bound of k-basis thm}
    Suppose that $D$ is a big $\QQ$-divisor of $k$-basis type for $k\gg 0$. Then
    \begin{equation*}
        a\leq \frac{1}{D^2}\int_0^{\tau(C)} \vol(D - uC)\dd u + \epsilon_k
    \end{equation*}
    where $\epsilon_k$ is a small constant depending on $k$ such that $\epsilon_k\to 0$ as $k\to \infty$.
\end{theorem}

Note that the above theorem is a generalisation of the definition of a cool divisor mentioned in Definition \ref{cool-divisor} and implies the following.

\begin{corollary}[{\cite[Corollary 2.10]{delta-inv-of-sing-dP}}]
    \label{Foundation:bound of k-basis cor}
    Suppose that $D$ is a big $\QQ$-divisor of $k$-basis type for $k\gg 0$, and 
    \[
        C\sim_{\QQ} \mu D     
    \]
    for some positive rational number $\mu$. Then 
    \[
        a\leq \frac{1}{3\mu} + \epsilon_k,    
    \]
    where $\epsilon_k$ is a small constant depending on $k$ such that $\epsilon_k \to 0$ as $k\to \infty$.
\end{corollary}

From the above results, it is evident that for a $k$-basis type divisor $D=aC+\Delta$, one can bound $a$ by computing the volume of the pseudoeffective divisor $D-uC$. Here, we recall how the $\mathrm{vol}(D)$ of an $\mathbb{R}$ divisor $D$ is calculated. If $D$ is not pseudoeffective, then $\mathrm{vol}(D)=0$. If $D$ is pseudoeffective, then there exists a nef $\mathbb{R}$ divisor $P$ on the surface $S$ such that 
$$
D \sim_{\RR} P+\sum_{i=1}^r a_iC_i,
$$ where each $C_i$ is an irreducible curve on $S$ with $P\cdot C_i=0$, each $a_i$ is a non-negative real number and the intersection form of the curves $C_1,C_2,...,C_r$ is negative definite. Such a decomposition is unique and $\mathrm{vol}(D)=\mathrm{vol}(P)=P^2$. Bounding the value of the coefficient $a$ enables us to use the inversion of adjunction argument (Lemma \ref{Foundation:mult}, Lemma \ref{Inversion of adjunction for singular point}) to then compute the value of $\delta_p(S)$ at various points $p \in S$.

Recently, there has been a lot of progress in understanding the applications of $\delta$-invariant. In this regard, we will briefly describe the theory developed by Abban (Ahmadinezhad) and Zhuang in \cite{AhmadinezhadZhuang}, particularly for surfaces, which we will use to prove K-stability of members of the family No. 2 in Theorem \ref{main theorem},  that is quasi-smooth, well-formed hypersurface of degree $2n+6$ in $\mathbb{P}(1,2,n+2,n+3)$ in Subsection \ref{subsection:No. 2}.

\subsubsection{Abban-Zhuang Theory}: 
Let $S$ be a del Pezzo surface that has at most Kawamata log terminal singularities,
let $\msp\in S$ be a point.
In this section, we will explicitly describe how to estimate $\delta_{\msp}(S)$, using the technique developed in \cite{AhmadinezhadZhuang}.

Let $Y$ be an irreducible curve in $S$ such that $\msp\in Y$. Assume that $Y$ is such that it is either Cartier on $S$ or $(S,Y)$ is purely log terminal. Then \cite[Theorem~3.3]{AhmadinezhadZhuang} implies the~following

\begin{theorem}
\cite[Theorem 1.7.1]{Fano3-folds}, \cite[Lemma~2.21]{AhmadinezhadZhuang},  \cite[Corollary~2.22]{AhmadinezhadZhuang}
\label{theorem:Hamid-Ziquan}
Let $S\left(W^Y_{\bullet,\bullet}; \msp\right)$ be the number defined in Theorem \ref{theorem:Kento-formula}. Then
$$
\delta_{\msp}(S)\geqslant\mathrm{min}\left\{\frac{1}{S_{S}(Y)},\frac{A_Y(\msp)}{S(W^Y_{\bullet,\bullet}; \msp)}\right\}.
$$
\end{theorem}

\begin{remark} \cite[Lemma 5.1.9]{IntrotoMMP} \cite[Proposition 16.6]{Asterisque}, \cite[Proposition 3.9]{Shokurov}  Recall that we have the following adjunction formula
$$
(K_S+Y) \vert_Y=K_Y+\Delta
$$ where $\Delta$ is called the \textit{different}. Then,
$$
A_Y(\msp)=1-\mathrm{coeff}(\Delta\vert_{\msp}).
$$ 
Since $K_S+Y$ is purely log terminal at $\msp$,
\begin{itemize}
    \item $\mathrm{coeff}(\Delta\vert_{\msp})=0$ when ${\msp}$ is a smooth point of $S$, 
    \item $\mathrm{coeff}(\Delta\vert_{\msp})=\frac{m-1}{m}$ when ${\msp}$ is a singular point of $S$ and is of type $\mathbb{C}^2 / \mathbb{Z}_m$, for $m\neq 1$. \end{itemize}
\end{remark}
The number $S(W^Y_{\bullet,\bullet}; \msp)$ is defined in detail in \cite{AhmadinezhadZhuang}. The following assertion helps in computing it.
\begin{theorem} 
\cite[Theorem 1.7.13]{Fano3-folds},\cite[Theorem 3.16]{Kento311}
\label{theorem:Kento-formula}
For every point $\msp$ over $Y$, let
$$
h(u)=\big(P(u)\cdot Y\big)\cdot\mathrm{ord}_{\msp}\Big(N(u)\big\vert_{Y}\Big)+\int_0^\infty \mathrm{vol}\big(P(u)\big\vert_Y-v\msp\big)\dd v,
$$ where
$P(u)$ and $N(u)$ are the positive and negative parts respectively of the Zariski Decomposition of the divisor $-K_S-uY$. Then one has \label{equation:Kento-formula-S}
$$
S\big(W^Y_{\bullet,\bullet}; \msp\big)=\frac{2}{(-K_S)^2}\int_0^\tau h(u)\dd u.
$$
\end{theorem}

\subsection{Valuative Criterion}

Let $S$ be a del Pezzo surface with at worst Kawamata log terminal singularities and let $G$ be a reductive subgroup of $\mathrm{Aut}(S)$. Let $f:\widetilde{S}\to S$ be a $G$-invariant birational morphism and let $E$ be a $G$-invariant prime divisor over $S$. Let
$$
\beta(E)=A_S(E)-S_{-K_S}(E),
$$
where $A_S(E)$ and $S_{-K_S}(E)$ are as defined in \eqref{A_S} and \eqref{S_S} respectively. \index{$\beta$-invariant}
We use the following criterion to prove the K-stability of members of family No.1 in Theorem \ref{main theorem}, that is, quasi-smooth, well-formed hypersurfaces of degree $2n$ in $\mathbb{P}(1,1,n,n)$.

\begin{theorem}\label{Zhuang} \cite[Corollary 4.14]{Z20}
Let $G$ be a reductive subgroup in $\mathrm{Aut}(S)$. Suppose $\beta(E)>0$ for every $G$-invariant dreamy prime divisor $E$ (see \cite[Definition 1.3]{F19} for the definition) over $S$. Then $S$ is $K$-polystable. 
\end{theorem}

\section{Proof of Main Theorem}
In this section, we prove the K-polystability of each of the members of families stated in Theorem \ref{main theorem} using various methods outlined in Section \ref{section:K-stability}. 

\subsection{$S_{2n}$ in $\mathbb{P}(1,1,n,n)$} In this section, we consider a quasi-smooth, well-formed hypersurface, $S_{2n}$ of degree $2n$ in $\mathbb{P}(1,1,n,n)$. After suitable changes of coordinates, we can assume that the surface $S_{2n}$ is given by 
$$
zt+f_{2n}(x,y)=0,
$$
where $f_{2n}(x,y)$ is a polynomial of degree $2n$ in $x$ and $y$. Since $S$ is quasi-smooth, $f_{2n}(x,y)$ does not have multiple roots.
We are going to prove that $S_{2n}$ is K-polystable for $n>1$, using Theorem \ref{Zhuang}. 

In our case, we can explicitly describe the action of the group $G=\mathbb{C}^* \times \mu_2$  on $S_{2n}$ as follows
\begin{align*}
    \lambda\cdot [x:y:z:t] & \longmapsto [x:y:\lambda z:\lambda^{-1} t] \quad\mathrm{for}\  \lambda \in \mathbb{C}^*,\\
    \mu_2\cdot[x:y:z:t] & \longmapsto [x:y:t:z].
\end{align*}

Here $[0:0:1:0]$ and $[0:0:0:1]$ are the two singular points of the surface $S_{2n}$ of type $\frac{1}{n}(1,1)$. Let $\mathcal{P}$ be the pencil of $G$-invariant curves in $S_{2n}$ generated by the curves $C_x$ and $C_y$. Then $-K_{S_{2n}} \sim 2C_x$ and $C_x^2=\frac{2}{n}$.

From the definition of the $\beta$-invariant of a prime divisor $E$ over $S_{2n}$, we have that 
$$
\beta(E)=A(E)-\frac{1}{(-K_S)^2}\int_0^\infty \mathrm{vol}(-K_S-uE)\dd u,
$$ where $A(E)$ is the log discrepancy of the divisor $E$.

So, we need to check $\beta(E)$ for all $E$ that are $G$-invariant prime divisors in $S_{2n}$ and all $G$-invariant prime divisors over $S_{2n}$.

All $G$-invariant irreducible prime divisors in $S_{2n}$ are irreducible curves in the pencil $\mathcal{P}$. Let $C$ be one such curve in $S_{2n}$, in the pencil $\mathcal{P}$, that is, $C:ax+by=0$ for some $[a:b] \in \mathbb{P}^1$. 
Consider the divisor $-K_S-xC$. Then we have that $\tau(C)=2$ and  $\sigma(C)=2$. Then 
\begin{align*}
    \beta(C)&=1-\frac{1}{4}\int_0^2 (2-u)^2\dd u=\frac{1}{3}.
\end{align*}
We now need to consider only the $G$-invariant prime divisors over $S_{2n}$. These divisors will be mapped to $G$-fixed points on $S_{2n}$. These $G$-fixed points are the singular points of the reducible curves in $\mathcal{P}$. Let us now understand what these $G$-fixed points are.

Observe that the reducible curves in $\mathcal{P}$ are cut out on $S_{2n}$ by the linear factors of the polynomial $f_{2n}(x,y)$. Since there are $2n$ such linear factors of the polynomial $f_{2n}(x,y)$, we have $2n$ reducible curves in $\mathcal{P}$.
After a suitable change of coordinates, we may assume that one such linear factor is the curve $C_x$ that is cut out by $x=0$, so that $f_{2n}=xg_{2n-1}$ for some polynomial $g_{2n-1}(x,y)$ of degree $2n-1$. This implies that the equation of $S$ looks like 
$$
zt+xg_{2n-1}(x,y)=0.
$$
Substituting for $x=0$, we see that we get reducible components $z=x=0;t=x=0$. So the singular point of the curve $C_x:x=0$, is $[0:1:0:0]$. Similarly, for each of the linear factors $h(x,y)$ of the polynomial $g_{2n-1}(x,y)$, we get the reducible components of the curve in $\mathcal{P}$ to be $z=h(x,y)=0$ and $t=h(x,y)=0$ and we can then compute the singular points in each case. In total, we get $2n$ singular points. These $2n$ points are all the $G$-invariant points on $S_{2n}$. 

\vspace{0.5cm}
Consider one such reducible curve and let the irreducible components of it be given by $C_1$ and $C_2$. The point of intersection of $C_1$ and $C_2$ is $[0:1:0:0]$. We will call this point $\msp$. Let 
$
\pi:\widetilde{S}_{2n} \to S_{2n}
$ be the blow up of $\msp$ with the exceptional divisor of the blow up denoted by $F.$
Observe that the action of $G$ lifts to $\widetilde{S}_{2n}$ and thus the the morphism $\pi$ is $G$-equivariant. 
Also observe that $G$ does not fix any point on $F$. Thus, the exceptional divisors upon blow up of the singular points of the reducible curves in the pencil $\mathcal{P}$, are the only irreducible $G$-invariant prime divisors over $S_{2n}$, and in this case, this is $F$.

Consider the divisor 
$$
    D=\pi^*(-K_{S_{2n}})-uF=2(\widetilde{C}_1+\widetilde{C}_2)+(4-u)F,
$$ 
where $\widetilde{C}_1$ and $\widetilde{C}_2$ are the proper transforms of $C_1$ and $C_2$ on $\widetilde{S}$ and the pseudoeffective threshold, $\tau(D)=4$. Let us now compute the Seshadri constant, $\sigma(D)$,
\begin{align*}
    D\cdot\widetilde{C}_1 = 2\left(\frac{1-2n}{n}\right)+(4-u)\geq 0.
\end{align*}
This implies that $\sigma(D)= \frac{2}{n}$. For $u\in [0,\frac{2}{n}]$, we have the following 
\begin{align*}
    \mathrm{vol}(D)=\mathrm{vol}(\pi^*(-K_{S_{2n}})-uF)&=\frac{8}{n}-u^2.
\end{align*}

For $u \in [\frac{2}{n},4]$, the Zariski decomposition is given by 
\begin{align*}
    D=\left(D-\left(\frac{2-nu}{1-2n}\right)\left(\widetilde{C}_1+\widetilde{C}_2\right)\right)+\left(\frac{2-nu}{1-2n}\right)\left(\widetilde{C}_1+\widetilde{C}_2\right).
\end{align*}
Therefore, 
\begin{align*}
    \mathrm{vol}(D)=\mathrm{vol}(\pi^*(-K_{S_{2n}})-uF)=(4-u)^2\left(\frac{1}{2n-1}\right).
\end{align*}

Thus, 
\begin{align*}
    \beta(F)=A(F)-\frac{1}{4C_x^2}\left(\int_0^{\frac{2}{n}}\left(\frac{8}{n}-u^2\right)\dd u+\int_{\frac{2}{n}}^4 (4-u)^2\frac{1}{2n-1}\dd u\right)=\frac{2n-2}{3n}.
\end{align*}
Therefore, by Theorem \ref{Zhuang}, for all $n>1$, this proves that all quasi-smooth hypersurfaces $S_{2n}$ of degree $2n$ in $\mathbb{P}(1,1,n,n)$ are K-polystable. 

\subsection{$S_{2n+6}$ in $\mathbb{P}(1,2,n+2,n+3)$}\label{subsection:No. 2}

Let $S_{2n+6}$ be a quasi-smooth, well-formed hypersurface in $\mathbb{P}(1,2,n+2,n+3)$ of degree $2n+6$. In this subsection, we will denote $S_{2n+6}$ with $S$. We may assume that the surface $S$ is given by 
$$
t^2+z^2y+f_{2n+6}(x,y)=0,
$$ where $f_{2n+6}(x,y)$ is a polynomial of degree $2n+6$ in $x$ and $y$.
When $n$ is even, the surface $S$ is singular at the points $\msp_z:=[0:0:1:0]$ of type $\frac{1}{n+2}(1,1)$ and $Q_1:=[0:1:\alpha:0];\ Q_2:=[0:1:\beta:0]$ of type $\frac{1}{2}(1,1),$ where $\alpha$ and $\beta$ are distinct real numbers and is smooth away from these points. Note that when $n=0$,the points $\msp_z$, $Q_1$ and $Q_2$  are of type $\frac{1}{2}(1,1).$

When $n$ is odd, the surface is singular at the points $\msp_t:[0:0:0:1]$ of type $\frac{1}{n+2}(1,1)$ and the points $Q_1:=[0:1:0:\alpha];\ Q_2:=[0:1:0:\beta]$ of type $\frac{1}{2}(1,1),$ where $\alpha$ and $\beta$ are distinct real numbers and is smooth away from these points. Note that all the singular points lie on the curve $C_x$.
We also have 
\[
-K_{S}^2=\frac{4}{n+2};\quad C_x^2=\frac{1}{n+2}.
\]
We use the stability threshold from Theorem \ref{theorem:Kento-Yuji-Blum-Xu} to conclude that the surface is K-stable for all $n \geq 0$. 

In order to prove the above claim, we will estimate the value of $\delta_{\msp}(S)$ using Theorem \ref{theorem:Hamid-Ziquan} for different positions of the point $\msp\in S$. In our case, we have $X=S$, $Z=\msp$, $L=-K_S$ and $Y$ is any curve that passes through the point $\msp$.

\begin{lemma}\label{PinCx}
    If $\msp\in C_x$, then $\delta_{\msp}(S)\geq \frac{3}{2}.$
\end{lemma}
\begin{proof}
Since $\msp \in C_x$ we will take $Y=C_x$ according to the notations of Theorem \ref{theorem:Hamid-Ziquan}. Note that $A_S(C_x)=1$. We will now compute $S_{-K_S}(C_x)$. Consider the divisor $-K_S-uC_x=(2-u)C_x$. With the help of intersection numbers of the divisor with $C_x$, we can conclude that $\tau(C_x)=2$ and that the divisor $-K_S-uC_x$ is nef when $u\in [0,2].$ We then have that 
\begin{align*}
    S_{-K_S}(C_x)&=\frac{1}{(-K_S)^2}\int_0^\infty \mathrm{vol}(-K_S-uC_x)\dd u,\\
    &=\frac{n+2}{4}\int_0^2(2-u)^2\dd u,\\
    &=\frac{2}{3}.
\end{align*}
Therefore, for any point $\msp$ on $C_x$,  $\frac{A_S(C_x)}{S_{-K_S}(C_x)}=\frac{3}{2}.
$
We will now use Theorem \ref{theorem:Kento-formula} to compute $S\big(W^{Y}_{\bullet,\bullet};\msp\big)$ for different points $\msp$ on $C_x$.

\textbf{Case 1:} Suppose the point $\msp=\msp_z$ on $C_x$. Recall that $\msp_z$ and $\msp_t$ are the singular points of type $\frac{1}{n+2}(1,1)$ when $n$ is even and odd respectively. Since both $\msp_z$ and $\msp_t$ lie on $C_x$, the proof below works for both. So without loss of generality, let us take $\msp=\msp_z$.  Note that $A_{C_x}(\msp)=\frac{1}{n+2}$. Therefore, using Theorem \ref{theorem:Kento-formula}, we get that 
\begin{align*}
S\big(W^{C_x}_{\bullet,\bullet}; \msp\big)&=\frac{2}{(-K_S)^2}\int_0^{\tau(C_x)} h(u) \dd u,
\end{align*} where 
$$
h(u)=\big(P(u)\cdot C_x\big)\cdot\mathrm{ord}_{\msp}\Big(N(u)\big\vert_{C_x}\Big)+\int_0^\infty \mathrm{vol}\big(P(u)\big\vert_{C_x}-v\msp\big)\dd v.
$$

Since the divisor $-K_s-uC_x=(2-u)C_x$ is nef when $u\in [0,\tau(C_x)=2]$, $P(u)=(2-u)C_x$ and $N(u)=0$. Therefore, 
\begin{align*}
    S\left(W^{C_x}_{\bullet,\bullet}; \msp\right)
    =\frac{n+2}{2}\int_0^2\int_0^{\frac{2-u}{n+2}}\left(\frac{2-u}{n+2}-v \right)\dd v\dd u
    =\frac{2}{3(n+2)}.
\end{align*}
This implies that 
$$
\frac{A_{C_x}(\msp)}{S(W^{C_x}_{\bullet,\bullet};\msp)}=\frac{\frac{1}{n+2}}{\frac{2}{3(n+2)}}=\frac{3}{2}.
$$
From Theorem \ref{theorem:Hamid-Ziquan}, we then get that for $\msp=\msp_z$,
$$
\delta_P(S)\geq \frac{3}{2}.
$$

\textbf{Case 2: } Suppose the point $\msp=Q_1$. Since $Q_1$ and $Q_2$ are singular points of type $\frac{1}{2}(1,1)$, the proof below works for both points. So without loss of generality, let us take $\msp=Q_1$. Note that $A_{C_x}(\msp)=\frac{1}{2}.$ Similar to Case 1, we get that $S\big(W^{C_x}_{\bullet,\bullet}; \msp\big)=\frac{2}{3(n+2)}$. This implies that 
$$
\frac{A_{C_x}(\msp)}{S(W^{C_x}_{\bullet,\bullet};\msp)}=\frac{\frac{1}{2}}{\frac{2}{3(n+2)}}=\frac{3(n+2)}{4}.
$$
From Theorem \ref{theorem:Hamid-Ziquan}, we then get that for $\msp=O_y$,
$$
\delta_{\msp}(S)\geq \frac{3}{2}.
$$

\textbf{Case 3: }Suppose the point $\msp \in C_x \setminus \{\msp_z, Q_1, Q_2\}$. This implies that the point $\msp$ is a smooth point. Then $A_{C_x}(\msp)=1$. Again, similar to Case 1, we get that $S\big(W^{C_x}_{\bullet,\bullet}; \msp\big)=\frac{2}{3(n+2)}$. This implies that 
$$
\frac{A_{C_x}(\msp)}{S(W^{C_x}_{\bullet,\bullet};\msp)}=\frac{1}{\frac{2}{3(n+2)}}=\frac{3(n+2)}{2}.
$$
From Theorem \ref{theorem:Hamid-Ziquan}, we then get that for $\msp\in C_x \setminus \{\msp_z, Q_1, Q_2\}$,
$$
\delta_{\msp}(S)\geq \frac{3}{2}.
$$

This completes proof of the claim.
\end{proof}
 
 \begin{lemma}\label{PnotinCx}
 If $\msp\in S\setminus C_x$, then $\delta_{\msp}(S)>1$.
 \end{lemma}
 \begin{proof}
 Since $\msp \notin C_x$, the point $\msp$ is a smooth point of the surface. Note that there is a curve from the pencil $C_{\mu}: y=\mu x^2$ for some $\mu \in \mathbb{R}$ that passes through the point $\msp$. 
 
 \textbf{Case 1:} Suppose the curve $C_{\mu}$ that passes through $\msp$ is irreducible. Note that $A_S(C_{\mu})=1$. Consider the divisor $-K_S-uC_{\mu}=(1-u)C_{\mu}.$ This implies that $\tau(C_{\mu})=1$ and that the divisor $-K_S-uC_{\mu}$ is nef when $u\in [0,1]$. We then have that 
 \begin{align*}
     S_{-K_S}(C_{\mu})&=\frac{1}{(-K_S)^2}\int_0^\infty\mathrm{vol}(-K_S-uC_{\mu})\dd u\\
     &=\int_0^1 (1-u)^2\dd u=\frac{1}{3}.
 \end{align*}
 Therefore, for any point $\msp\in C_{\mu}$,$ \frac{A_S(C_{\mu})}{S_{-K_S}(C_{\mu})}=3.
 $
 
 Let us now try to bound $\delta_{\msp}(S)$ using Theorem \ref{theorem:Hamid-Ziquan}. According to the notations of Theorem \ref{theorem:Hamid-Ziquan}, $Y=C_\mu$. Note that $A_{\msp}(C_{\mu})=1$ since $\msp$ is a smooth point.  When $u\in [0,\tau(C_{\mu})=1]$, $P(u)=(1-u)C_{\mu}$ and $N(u)=0$ since the divisor $-K_S-uC_{\mu}=(1-u)C_{\mu}$ is nef. Therefore, 
 \begin{align*}
    S\left(W^{C_\mu}_{\bullet,\bullet}; \msp\right) =\frac{n+2}{2}\int_0^1 \int_0^{\frac{4(1-u)}{n+2}} \left(\frac{4(1-u)}{n+2}-v\right) \dd v\dd u =\frac{4}{3(n+2)}.
\end{align*}
This implies that 
$$
    \frac{A_{C_\mu}(\msp)}{S(W^{C_\mu}_{\bullet,\bullet};\msp)} = \frac{1}{\frac{4}{3(n+2)}}=\frac{3(n+2)}{4}.
$$
From Theorem \ref{theorem:Hamid-Ziquan}, we then get that for $\msp\in C_{\mu}\setminus C_x$ where $C_{\mu}$ is irreducible,
$$
    \delta_{\msp}(S)\geq  \mathrm{min}\left\{3,\frac{3(n+2)}{4}\right\}>1,
$$ for all $n\geq 0$.

\textbf{Case 2: }Suppose the curve $C_{\mu}$ is reducible with $C_\mu=Z_1+Z_2$ where $Z_1$ and $Z_2$ are the irreducible components of the curve $C_\mu$. 

Note that $A_S(Z_1)=1$. Consider the divisor $-K_S-uZ_1=(1-u)Z_1+Z_2$. Since $Z_1^2=Z_2^2=\frac{-(n+1)}{n+2}$, we can conclude that $\tau(Z_1)=1.$ From intersections of the divisor with $Z_1$ and $Z_2$ we get that the divisor $-K_S-uZ_1$ is nef when $u\in [0,\frac{2}{n+3}]$.
When $u\in [\frac{2}{n+3},1]$, the Zariski decomposition is given by 
$$
-K_S-uZ_1=\bigg((1-u)Z_1+\frac{(1-u)(n+3)}{n+1}Z_2\bigg)+\frac{u(n+3)-2}{n+1}Z_2.
$$ 
Therefore,
\begin{equation*}
    \mathrm{vol}(-K_S-uZ_1)=
    \begin{cases}
        \dfrac{4(1-u)-(n+1)u^2}{n+2} & \mathrm{if\ } u\in\left[0,\dfrac{2}{n+3}\right], \\[10pt]
        \dfrac{4(u-1)^2}{n+1} & \mathrm{if\ } u\in\left[\dfrac{2}{n+3},1\right].
    \end{cases}
\end{equation*}
Using this, we get that 
\begin{align*}
    S_{-K_S}(Z_1)&=\frac{1}{(-K_S)^2}\int_0^\infty\mathrm{vol}(-K_S-uZ_1)\dd u\\
    &=\frac{n+2}{4}\left(\int_0^{\frac{2}{n+3}}\frac{4(1-u)-(n+1)u^2}{n+2}\dd u+\int_{\frac{2}{n+3}}^1 \frac{4(u-1)^2}{n+1}\dd u\right)\\
    &=\frac{n+4}{3(n+3)}.
\end{align*}
Therefore, for any point $\msp\in Z_1$,
 $ \frac{A_S(Z_1)}{S_{-K_S}(Z_1)}=\frac{3(n+3)}{n+4}.
 $

\textbf{Case 2a: }Suppose the point $\msp$ is on one of the components. Without loss of generality, let us assume that $\msp\in Z_1 \setminus Z_2$.

We will now bound $\delta_{\msp}(S)$ using Theorem \ref{theorem:Hamid-Ziquan}. In this case, we will take $Y=Z_1$. Note that $A_{Z_1}(\msp)=1$. Recall that we have 
$$-K_S-uZ_1=P(u)+N(u)$$ where
\begin{equation*}
    P(u)=
    \begin{cases}
        (1-u)Z_1+Z_2 & \mathrm{if\ } u\in\left[0,\dfrac{2}{n+3}\right],\\[10pt]
        (1-u)Z_1+\dfrac{(1-u)(n+3)}{n+1}Z_2 & \mathrm{if\ } u\in\left[\dfrac{2}{n+3},1\right]
    \end{cases}
\end{equation*}
and
\begin{equation*}
    N(u)=
    \begin{cases}
        0& \mathrm{if\ } u\in\left[0,\dfrac{2}{n+3}\right],\\[10pt]
        \dfrac{u(n+3)-2}{n+1}Z_2  & \mathrm{if\ } u\in\left[\dfrac{2}{n+3},1\right].
    \end{cases}
\end{equation*}
Therefore, 
\begin{align*}
    S\left(W^{Z_1}_{\bullet,\bullet}; \msp\right) &=\frac{n+2}{2}\left[\int_0^\frac{2}{n+3}\int_0^{\frac{2+u(n+1)}{n+2}}\bigg(\frac{2+u(n+1)}{n+2}-v \bigg)\dd v \dd u\right.\\
    &\left.+\int_{\frac{2}{n+3}}^1\int_0^{\frac{4(1-u)}{n+1}}\bigg(\frac{4(1-u)}{n+1}-v\bigg) \dd v \dd u\right]\\ &=\frac{2(2n+5)}{3(n+2)(n+3)}.
\end{align*}
This implies that 
$$
\frac{A_{Z_1}(\msp)}{S(W^{Z_1}_{\bullet,\bullet};\msp)}=\frac{1}{\frac{2(2n+5)}{3(n+2)(n+3)}}=\frac{3(n+2)(n+3)}{2(2n+5)}.
$$
From Theorem \ref{theorem:Hamid-Ziquan}, we then get that for $\msp\in (Z_1\cup Z_2)\setminus C_x$ where $Z_1,Z_2$ are irreducible components of the curve $C_\mu$,
$$
\delta_{\msp}(S)\geq \mathrm{min}\left\{\frac{3(n+3)}{n+4},\frac{3(n+2)(n+3)}{2(2n+5)}\right\}=\frac{3(n+2)(n+3)}{2(2n+5)}>1,
$$for all $n\geq 0$.

\textbf{Case 2b: }Suppose the point $\msp$ is on both components of the curve $C_{\mu}$, i.e. $\msp\in Z_1 \cap Z_2$.

Let us consider the divisor $-K_S-uZ_1$. From the computations in Case 2a, we see that 
$
\frac{A_S(Z_1)}{S_{-K_S}(Z_1)}=\frac{3(n+3)}{n+4}.
$
We will now bound $\delta_{\msp}(S)$ using Theorem \ref{theorem:Hamid-Ziquan}. In order to use this theorem, let $Y=Z_1$ by the notations of Theorem \ref{theorem:Hamid-Ziquan}. Note that $A_{Z_1}(\msp)=1$. From the Zariski decompositions calculated earlier for the divisor $-K_S-uZ_1$, we can compute the following.
\begin{align*}
    S\big(W^{Z_1}_{\bullet,\bullet}; \msp\big) &=\frac{n+2}{2}\Bigg[\int_0^\frac{2}{n+3}\int_0^{\frac{2+u(n+1)}{n+2}}\bigg(\frac{2+u(n+1)}{n+2}-v \bigg)\dd v \dd u\\&+\int_{\frac{2}{n+3}}^1 \bigg(\frac{4(1-u)(u(n+3)-2)(n+3)}{(n+1)^2(n+2)}\bigg)\dd u\\ 
    &+\int_{\frac{2}{n+3}}^1\int_0^{\frac{4(1-u)}{n+1}}\bigg(\frac{4(1-u)}{n+1}-v\bigg) \dd v \dd u\Bigg]\\ &=\frac{n+4}{3(n+2)}.
\end{align*}
This implies that 
$$
\frac{A_{Z_1}(\msp)}{S(W^{Z_1}_{\bullet,\bullet};\msp)}=\frac{1}{\frac{n+4}{3(n+2)}}=\frac{3(n+2)}{n+4}.
$$
From Theorem \ref{theorem:Hamid-Ziquan}, we then get that for $\msp\in (Z_1\cap Z_2)\setminus C_x$ where $Z_1,Z_2$ are irreducible components of the curve $C_\mu$,
$$
\delta_{\msp}(S)\geq \mathrm{min}\left\{\frac{3(n+3)}{n+4},\frac{3(n+2)}{n+4}\right\}=\frac{3(n+2)}{n+4}>1,
$$ for all $n\geq 0$.
This completes the proof of the claim.
 \end{proof}
\begin{proof}[Proof of Main Theorem]
 From Lemma \ref{PinCx} and Lemma \ref{PnotinCx}, we can conclude that 
 \begin{align*}
\delta(S)&=\inf_{\msp\in S}\delta_{\msp}(S)>1.
\end{align*} Therefore, Theorem \ref{theorem:Kento-Yuji-Blum-Xu} implies that $S$ is K-stable for all $n\geq 0$.
\end{proof}

\subsection{$S_{12}$ in $\mathbb{P}(1,3,4,6)$}

Let $S_{12}$ be a quasi-smooth hypersurface in $\mathbb{P}(1,3,4,6)$ of degree $12$. By a suitable coordinate change we can assume that $S_{12}$ can be given by a quasihomogeneous polynomial

\begin{equation*}
    t^2 + z^3 + y^4 + xf(x,y,z) = 0
\end{equation*}
where $f(x,y,z)$ is a quasihomogeneous polynomial in $x,y,z$ of degree $11$. The surface $S_{12}$ is singular at the points $[0:i:0:1]$, $[0:-i:0:1]$ (of type $\frac{1}{3}(1,1)$) and $[0:0:-1:1]$ (of type $\frac{1}{2}(1,1)$). Note that all the singular points of $S_{12}$ lie on the curve $C_x$. Since $C_x$ is isomorphic to the variety given by $$ t^2+z^3+y^4=0 $$ in $\mathbb{P}(3,4,6)$ we can check that it is irreducible.

Let $D$ be an anticanonical $\QQ$-divisor of $k$-basis type on $S_{12}$ with $k\gg 0$. We set $\lambda=\frac{6}{5}$.

\begin{proposition} \label{Prop3}
    The log pair $(S_{12},\lambda D)$ is log canonical.
\end{proposition}

We will now prove Proposition \ref{Prop3}. 
\begin{lemma}
    The log pair $(S_{12}, \lambda D)$ is log canonical along $S_{12}\setminus C_x$.
\end{lemma}

\begin{proof}
    Suppose that the log pair $(S_{12}, \lambda D)$ is not log canonical at some point $\msp = [p_1 : p_2 : p_3 : p_4]\in S_{12}\setminus C_x$. By the coordinate change 
    \begin{equation*}
        G(x,y,z,t) = F\left(x, y + \frac{p_2}{p_1^3}x^3, z + \frac{p_3}{p_1^4}x^4, t + \frac{p_4}{p_1^6}x^6\right)
    \end{equation*}
    we can assume that the surface $S_{12}$ is given by the quasihomogeneous polynomial $G(x,y,z,t) = 0$ and $\msp = [1:0:0:0]$. In the chart $U_x$ defined by $x = 1$ it is given by 
    \begin{align*}
        G(1,y,z,t) &= a_1y + a_2z + a_3t + a_4y^2 + a_5yz + a_6yt + a_7z^2 + a_8zt + t^2\\
        &+ a_9 y^3 + a_{10}y^2z + a_{11}y^2t + a_{12}yz^2 + z^3 + y^4 = 0
    \end{align*}
    where $a_1,\ldots, a_{12}$ are constants.

    We consider the linear terms of $G(1,y,z,t)$. We first assume that $a_2 = a_3 = 0$. Then $C_y$ is isomorphic to the variety given by
    \begin{equation*}
        a_7x^4z^2 + a_8x^2zt + t^2 + z^3 = 0
    \end{equation*}
    in $\PP(1,4,6)$. Then $\mult_{\msp}(C_y) = 2$ and $C_y$ is irreducible. Write
    \[
        D = \zeta_y C_y + \Psi_y
    \]
    where $\zeta_y$ is the nonnegative constant and $\Psi_y$ is an effective $\QQ$-divisor such that $C_y\not\subset \Supp(\Psi_y)$. By Lemma \ref{Foundation:inequality mult intersection} we have
    \[
        2\left(\frac{1}{\lambda} - 2\zeta_y\right) < \mult_{\msp}(C_y)\mult_{\msp}(D - \zeta_y C_y) \leq  C_y\cdot(D - \zeta_y C_y) = 1 - \frac{3}{2}\zeta_y.
    \]
    It implies that $\frac{4}{15} < \zeta_y$. Meanwhile, by Corollary \ref{Foundation:bound of k-basis cor} we have $\zeta_y \leq \frac{2}{9} + \epsilon_k$. It is impossible. Thus $a_2\neq 0$ or $a_3\neq 0$.
    
    Next we assume that $a_2\neq 0$ and $a_3 = 0$. Let $C$ be the curve given by $a_1xy + a_2z = 0$. Then $\mult_{\msp}(C) = 2$. We first consider the case that $C$ is irreducible. Write
    \[
        D = \zeta C + \Psi
    \]
    where $\zeta$ is a constant and $\Psi$ is effective $\QQ$-divisor such that $C\not\subset \Supp(\Psi)$. By Lemma \ref{Foundation:inequality mult intersection} we have
    \[
        2\left(\frac{1}{\lambda} - 2\zeta\right) < \mult_{\msp}(C)\mult_{\msp}(D - \zeta C) \leq C\cdot(D - \zeta C) = \frac{4}{3} - \frac{8}{3}\zeta.
    \]
    It implies that $\frac{1}{4} < \zeta$. Meanwhile, by Corollary \ref{Foundation:bound of k-basis cor} we have $\zeta \leq \frac{1}{6} + \epsilon_k$. It is impossible. Thus $C$ must be reducible. Then it is given by
    \begin{equation*}
        a_1xy + a_2z = (t + b_1y^2 + b_2yx^3)(t + b_3y^2 + b_4yx^3) = 0
    \end{equation*}
    where $b_1,\ldots, b_4$ are constants. Write
    \begin{equation*}
        C = L + R
    \end{equation*}
    where $L$ and $R$ are the curves that are given by $a_1xy + a_2z = t + b_1y^2 + b_2yx^3 = 0$ and $a_1xy + a_2z = t + b_3y^2 + b_4yx^3 = 0$ in $\PP(1,3,4,6)$, respectively. We have the following intersection numbers:
    \[
        L\cdot R = 2,\qquad L^2 = R^2 = -\frac{2}{3},\qquad -K_{S_{12}}\cdot L = -K_{S_{12}}\cdot R = \frac{2}{3}.
    \]
    We write
    \[
        D = lL + \Omega
    \]
    where $l$ is non-negative constant and $\Omega$ is an effective $\QQ$-divisor such that $L\not\subset \Supp(\Omega)$. From Theorem \ref{Foundation:bound of k-basis thm} we can obtain
    \[
        l\leq \frac{1}{(-K_{S_{12}})^2}\int_{0}^{\tau(L)} \vol(-K_{S_{12}} - uL) \dd u + \epsilon_k
    \]
    where $\epsilon_k$ is a small constant depending on $k$ such that $\epsilon_k \to 0$ as $k\to \infty$. Since
    \[
        -K_{S_{12}} -uL\equiv \left(\frac{1}{2} -u\right)L + \frac{1}{2}R
    \]
    and $R^2 < 0$, we have $\vol(-K_{S_{12}} - uL) = 0$ for $u\geq \frac{1}{2}$. Thus $\tau(L) = \frac{1}{2}$. We have
    \[
        (-K_{S_{12}} - uL)\cdot R = \frac{2}{3} - 2u.
    \]
    It implies that $-K_{S_{12}} - uL$ is nef for $u\leq \frac{1}{3}$. Moreover the Zariski decomposition of $-K_{S_{12}} - uL$ is 
    \[
        \left(\frac{1}{2} - u\right)\left(L + 3R\right) + (-1 + 3u)R
    \]
    for $\frac{1}{3}\leq u \leq \frac{1}{2}$ where the left part is the positive part. Thus we have
    \[
        \vol(-K_{S_{12}} - uL) = \left\{
        \begin{array}{ll}
            (-K_{S_{12}} - uL)^2 = \frac{2}{3} - \frac{4}{3}u - \frac{2}{3}u^2\qquad \textrm{for~} u\leq \frac{1}{3},\\ \\
            \left(\frac{1}{2} - u\right)^2\left(L + 3R\right)^2 = \frac{28}{3}(\frac{1}{2}-u)^2 \qquad \textrm{for~} \frac{1}{3}\leq u \leq \frac{1}{2}.
        \end{array}\right.
    \]
        From this we have
    \[
        \int_{0}^{\tau(L)} \vol(-K_X - uL)\dd u = \int_{0}^{\frac{1}{3}}\frac{2}{3} - \frac{4}{3}x - \frac{2}{3}x^2 \dd u + \int_{\frac{1}{3}}^{\frac{1}{2}} \frac{28}{3}\left(\frac{1}{2}-x\right)^2 \dd u = \frac{25}{162}.
    \]
    Thus we obtain a bound $l\leq \frac{25}{108} + \epsilon_k$. Since $\lambda l\leq 1$ the log pair $(S_{12}, L + \lambda\Omega)$ is not log canonical at $\msp$. By Corollary \ref{Inversion of adjunction for smooth point} we have
    \[
        \frac{1}{\lambda} < \mult_{\msp}(\Omega|_L) \leq \Omega\cdot L = (D - lL)\cdot L = \frac{2}{3} + \frac{2}{3}l.
    \]
    Then we have $l > \frac{1}{4}$. It is impossible. Thus $a_3\neq 0$.

    Finally we assume that $a_3\neq 0$. Then $y$ and $z$ are local coordinates at $\msp$. Let $\phi\colon Y\to X$ be a blow-up at $\msp$. Then we have
    \[
        \phi^*(D) \equiv \widetilde{D} + \eta E
    \]
    where $\widetilde{D}$ is the strict transform of $D$, $\eta$ is a constant and $E$ is the exceptional divisor of $\phi$. We have
    \[
        K_Y + \lambda \widetilde{D} + (\lambda\eta - 1)E \equiv \phi^*(K_X + \lambda D).
    \]
    By the above equation the log pair $(Y, \lambda \widetilde{D} + (\lambda\eta - 1)E)$ is not log canonical at some point $\msq\in E$. There is the effective divisor $M$ given by $\zeta_1 xy + \zeta_2 z = 0$ such that $\widetilde{M}$ passes through the point $\msq$, where $\zeta_1$ and $\zeta_2$ are constants, where $\widetilde{M}$ is the strict transform of $M$. Then there are two cases. First we assume that $M$ is reducible. Then $M$ is the sum of two curves $R_1$ and $R_2$ that are given by $\zeta_1xy + \zeta_2z = t + c_1y^2 + c_2yx^3 = 0$ and $\zeta_1xy + \zeta_2z = t + c_3y^2 + c_4yx^3 + c_5x^6 = 0$ in $\PP(1,3,4,6)$, respectively. In this situation a procedure to obtain a contradiction is similar to the case that $\msp\in L$. Thus we can see that $M$ is irreducible.
    
    We write 
    \[
        D = mM + \Delta
    \]
    where $m$ is a non-negative constant and $\Delta$ is an effective $\QQ$-divisor such that $M\not\subset \Supp(\Delta)$. By Corollary \ref{Foundation:bound of k-basis cor} we have $\lambda m\leq 1$. Thus the log pair $(Y, \widetilde{M} + \lambda \widetilde{\Delta} + (\lambda \eta - 1)E)$ is not log canonical at $\msq$. By Corollary \ref{Foundation:mult} and the inversion of adjunction formula we have the inequality
    \[
        1 < \mult_{\msq}((\lambda \widetilde{\Delta} + (\lambda \eta - 1)E)|_{\widetilde{M}})\leq (\lambda \widetilde{\Delta} + (\lambda \eta - 1)E)\cdot \widetilde{M}.
    \]
    We have the following intersection numbers.
    \[
        \widetilde{\Delta}\cdot \widetilde{M} = (\widetilde{D} - a\widetilde{M})\cdot \widetilde{M} = \frac{4}{3} - \eta - \frac{5}{3}m,\qquad E\cdot \widetilde{M} = 1.
    \]
    Then the above inequality implies that
    \[
        1 < \lambda\left(\frac{4}{3} - \eta - \frac{5}{3}m\right) + (\lambda \eta - 1).
    \]
    Thus we have $-\frac{1}{5} > m$. It is impossible.
\end{proof}

We write 
\begin{equation*}
    D = aC_x + \Delta
\end{equation*}
where $a$ is a non-negative constant and $\Delta$ is an effective $\QQ$-divisor such that $C_x\not\subset \Supp(\Delta)$. By Corollary \ref{Foundation:bound of k-basis cor} we have 
$a \leq \frac{2}{3} + \epsilon_k$ where $\epsilon_k$ is a small constant depending on $k$ such that $\epsilon_k\to 0$ as $k\to \infty$.
\begin{lemma}
    The log pair $(S_{12}, \lambda D)$ is log canonical along $C_x\setminus \Sing(S_{12})$.
\end{lemma}
\begin{proof}
Suppose the log pair $(S_{12}, \lambda D)$ is not log canonical at some point $\msp\in C_x\setminus\Sing(S_{12})$. Note that this is a smooth point of the surface $S_{12}$.

Since $\lambda a <1$, applying Lemma \ref{Inversion of adjunction for smooth point} to $C_x$, we get
\begin{equation*}
  \frac{1}{\lambda}<\mult_{\msp}(C_x\cdot \Delta) \leq C_x \cdot \Delta = \frac{1}{3} - \frac{1}{6}a.
\end{equation*}
This implies that $a<0$ and therefore this is impossible.
\end{proof}

\begin{lemma}
    The log pair $(S_{12}, \lambda D)$ is log canonical at $\msp_z\in\mathrm{Sing}(S_{12})$.
\end{lemma}
\begin{proof} The point $\msp_z$ is a singular point of $S_{12}$ of type $\frac{1}{2}(1,1)$. Suppose that $(S_{12}, \lambda D)$ is not log canonical at $\msp_z$. Since $\lambda a<1$ and $C_x$ is smooth at $\msp_z$, using Lemma \ref{Inversion of adjunction for singular point}, we have
        \begin{equation*}
            \frac{1}{2\lambda}<\mult_{\msp}(C_x\cdot \Delta) \leq (C_x \cdot \Delta) = \frac{1}{3} - \frac{1}{6}a. 
        \end{equation*}This implies that $a<0$ which is impossible, thus proving our claim.
\end{proof}

\begin{lemma}
    The log pair $(S_{12}, \lambda D)$ is log canonical at $\msp_y\in\mathrm{Sing}(S_{12})$.
\end{lemma}

\begin{proof}
    Without loss of generality, consider the point $\msp_y$ which is a singular point of $S_{12}$ of type $\frac{1}{3}(1,1)$. Suppose that $(S_{12}, \lambda D)$ is not log canonical at $\msp_y$.

    Let $\pi_1\colon \widetilde{S}_{12} \to S_{12}$ be the weighted blow-up at $\msp_y$ with weight $(1,1) = (x,z)$. Then we have
\begin{equation*}
        K_{\widetilde{S}_{12}} \equiv \pi_1^*(K_{S_{12}}) - \dfrac{1}{3}E,\qquad
        \widetilde{D}\equiv\pi_1^*(D)-mE,
        \qquad
        \widetilde{C}_{x} \equiv \pi_1^*(C_{x})-\dfrac{1}{3}E
\end{equation*}
where $E$ is the exceptional divisor of $\pi$ and $m$ is a non-negative constant, and $\widetilde{C}_x$ and $\widetilde{\Delta}$ denote the strict transforms of $C_x$ and $\Delta$ on $\widetilde{S}_{12}$, respectively. Note that we denote $\widetilde{D}=a\widetilde{C}_x+\widetilde{\Delta}$.

The log pullback is given by 
\begin{equation*}
  K_{\widetilde{S}_{12}} + \lambda a\widetilde{C}_x + \lambda\widetilde{\Delta} + \left(\lambda m+\frac{1}{3}\right)E\equiv \pi_1^*(K_{S_{12}} + \lambda D).
\end{equation*} Thus, we have that $(\widetilde{S}_{12},\lambda a\widetilde{C}_x+\lambda\widetilde{\Delta} + \left(\lambda m+\frac{1}{3}\right)E)$ is not log canonical at a point $\msq\in E$. Let us bound the value of $m$. 

Consider
\begin{equation*}
    \pi_1^*(-K_{S_{12}}) - uE\equiv \pi_1^*(2C_x) - uE\equiv 2\widetilde{C}_x+\left(\frac{2}{3}-u\right)E.
\end{equation*}
Since $\widetilde{C}_x^2 = -\frac{1}{6} < 0$, $\tau(E) = \frac{2}{3}$. This divisor is nef when $u\in [0,\frac{1}{3}]$. When $u\in [\frac{1}{3},\frac{2}{3}]$, we have the following Zariski decomposition
$$
\pi_1^*(-K_{S_{12}}) - uE=\left((4-6x)\widetilde{C}_x+\left(\frac{2}{3}-u\right)E\right)+(6u-2)\widetilde{C}_x.
$$

Therefore,
\begin{equation*}
    \vol(\pi_1^*(-K_{S_{12}}) - uE) = 
    \begin{cases}
        (\pi_1^*(-K_{S_{12}}) - uE)^2 = \dfrac{2}{3} -3u^2 & \textrm{for~} 0\leq u \leq \dfrac{1}{3}\\[10pt]
        \left(\dfrac{2}{3} - u\right)^2(6\widetilde{C}_x + E)^2 = 3\left(\dfrac{2}{3} - u\right)^2 & \textrm{for~} \dfrac{1}{3} \leq u \leq \dfrac{2}{3}.
    \end{cases}
\end{equation*}
Then, by Theorem \ref{Foundation:bound of k-basis thm},
\begin{align*}
   m &\leq \frac{1}{(-K_{S_{12}})^2}\int_0^{\frac{2}{3}}\vol(\pi_1^*(-K_{S_{12}}) - uE) \dd u+\epsilon_k\\
   &= \frac{3}{2}\left(\int_0^{\frac{1}{3}}\frac{2}{3} - 3u^2\dd u + \int_{\frac{1}{3}}^{\frac{2}{3}} 3\left(\frac{2}{3} - u\right)^2 \dd u\right) + \epsilon_k\\
   &= \frac{1}{3}+\epsilon_k.
\end{align*}
This implies that $\lambda m+\frac{1}{3}<1$. We will now look at various possible positions of the point $\msq$. 
Suppose that $\msq\in \widetilde{C}_x$. Since $\lambda a < 1$, by Lemma \ref{Inversion of adjunction for smooth point} we have
\begin{equation*}
  1<\left(\lambda\widetilde{\Delta} + \left(\lambda m +\frac{1}{3} \right)E\right)\cdot \widetilde{C}_x = \lambda\left(\frac{1}{3} + \frac{a}{6}\right) + \frac{1}{3}.
\end{equation*}
It implies that $\lambda (a + 2) > 4$ and this is a contradiction. Thus $\msq\notin \widetilde{C}_x$. This implies that the log pair
\begin{equation*}
  \left(\widetilde{S}_{12}, \lambda\widetilde{\Delta} + \left(\lambda m+\frac{1}{3}\right)E \right)
\end{equation*}
is not log canonical at $\msq\in E\setminus \widetilde{C}_x$.

Let $\pi_2\colon \widehat{S}_{12} \to \widetilde{S}_{12}$ be the blow-up at $\msq$ with exceptional divisor $F$. Let $\widehat{\Delta}$ and $\widehat{E}$ denote the strict transforms of $\Delta$ and $E$, in $\widehat{S}_{12}$ respectively. Then we have
\begin{equation*}
    K_{\widehat{S}_{12}} \equiv\pi^*(K_{S_{12}})-\frac{1}{3}\widehat{E}+\frac{2}{3}F,\quad
    \widehat{D}\equiv\pi^*(D)-m\widehat{E}-(m+n)F,\quad
    \widehat{C}_x\equiv\pi^*(C_x)-\frac{1}{3}\widehat{E}-\frac{1}{3}F,
\end{equation*}
where $n=\mult_{\msq}(\widetilde{D})$ and $\pi=\pi_1 \circ \pi_2$.

Since 
$$
0\leq \widehat{\Delta}\cdot \widehat{E}=3m-n-a,
$$ we have that $\lambda (m+n)-\frac{2}{3}=\lambda(4m-a)-\frac{2}{3}<1$. Therefore, the log pair 
\begin{equation*}
  \left(\widehat{S}_{12},\lambda a\widehat{C}_x+ \lambda\widehat{\Delta}+\left(\lambda m+ \frac{1}{3}\right)\widehat{E} +  \left(\lambda (m+n)-\frac{2}{3}\right)F\right)
\end{equation*}
is not log canonical at some point $\mst\in F$.  We will now look at different positions of the point $\mst$. If $\mst\in \widehat{E}\cap F$ then applying Lemma \ref{Inversion of adjunction for smooth point} to $\widehat{E}$, we get 
\begin{equation*}
  1<\left(\lambda \widehat{\Delta}+\left(\lambda(m+n)-\frac{2}{3}\right)F\right)\cdot \widehat{E}=\lambda(4m-a)-\frac{2}{3}.
\end{equation*}
This implies that  $m > \frac{5}{12\lambda} + \frac{a}{4}$. This is a contradiction and thus $\mst\in F\setminus \widehat{E}$. Then the log pair 
\begin{equation*}
  \left(\widehat{S}_{12}, \lambda\widehat{\Delta} +  \left(\lambda(m+n)-\frac{2}{3}\right)F\right)
\end{equation*}
is not log canonical at $\mst$. Since $\left(\lambda(m+n)-\frac{2}{3}\right)<1$, using Lemma \ref{Inversion of adjunction for smooth point}, we get 
\[
\frac{1}{\lambda}<\widehat{\Delta}\cdot F=n.
\]

Consider the linear system $\mathcal{L}$ generated by $\alpha(\zeta_1x^3y + \zeta_2x^2z + t) + \beta x^6 = 0$ where $[\alpha:\beta]\in \PP^1$. Since the base locus of $\mcL$ is a finite points set, there is a member $M\in \mcL$ such that $M\not\subset \Supp(D)$. Let $D=\nu M+\Delta$ where $\Delta$ is an effective divisor such that $M \not\subset \mathrm{Supp}(\Delta)$ and $\nu$ is a non-negative constant. Note that $\msp_y\in M$ and we can choose $M$ such that $\widetilde{M}\in \widetilde{S}_{12}$ passes through the point $\msq \in E\setminus \widetilde{C}_x$ and $\widehat{M} \in \widehat{S}_{12}$ passes through the point $\mst\in F \setminus \widehat{E}$.
We then have
\begin{equation*}
    0\leq \widehat{\Delta}\cdot \widehat{M} = \widetilde{\Delta}\cdot \widetilde{M} - 2n = (\widetilde{D} - \widetilde{C}_x)\cdot \widetilde{M} -2n = 2 - 3m - 2n.
\end{equation*}
This implies that $2 > \frac{2}{\lambda} + 3m$. However we have $m > \frac{1}{3\lambda}$. It is impossible, thus proving our claim. 

\end{proof}

This completes the proof of Proposition \ref{Prop3}, which implies the following.
\begin{corollary}
$\delta(S_{12})\geq \frac{6}{5}.$
\end{corollary}

\subsection{$S_{15}$ in $\PP(1,4,5,7)$}

Let $S_{15}$ be a quasi-smooth hypersurface in $\PP(1,4,5,7)$ of degree $15$. By suitable change of coordinates, we can assume that $S_{15}$ is given by a quasihomogeneous polynomial
\begin{equation*}
  t^2x + ty^2 + z^3 + xf(x,y,z) = 0
\end{equation*}
where $f(x,y,z)$ is a quasihomogeneous polynomial of degree $14$. The surface $S_{15}$ is singular at $\msp_y:=[0:1:0:0]$ of type $\frac{1}{4}(1,1)$ and at $\msp_t:=[0:0:0:1]$ of type $\frac{1}{7}(4,5)$. Observe that both singular points lie on the curve $C_x$. Note that $C_x$ is isomorphic to the variety given by $$ty^2+z^3=0$$ in $\PP(4,5,7)$. This shows that the affine cone of the curve $C_x$ is only singular at $\msp_t$ except the vertex and is smooth at $\msp_y$. 

Let $D$ be an anticanonical $\QQ$-divisor of $k$-basis type on $S_{15}$ with $k\gg 0$. We write 
\begin{equation*}
    D = a C_x + \Delta
\end{equation*}
where $a$ is non-negative rational number and $\Delta$ is an effective $\QQ$-divisor such that $C_x\not\subset \Supp(\Delta)$. By Corollary \ref{Foundation:bound of k-basis cor} we have $a\leq \frac{3}{4}$. We set $\lambda = \frac{11}{10}$.

\begin{proposition}\label{Proposition: 4}
    The log pair $(S_{15}, \lambda D)$ is log canonical.
\end{proposition}

Let us now prove Proposition \ref{Proposition: 4}.
\begin{lemma}
    The log pair $(S_{15}, \lambda D)$ is log canonical along $S_{15}\setminus C_x$.
\end{lemma}
\begin{proof}
    Suppose the log pair $(S_{15}, \lambda D)$ is not log canonical at some point $\msp \in S_{15}\setminus C_x$. Then the point $\msp$ is a smooth point of the surface $S_{15}$. There is a divisor $M\in |\mcO_{S_{15}}(4)|$ passing through $\msp\in M$. Then $M$ is given by the hyperplane section of $y-\zeta x^4 = 0$ in $S_{15}$ where $\zeta$ is a constant. It is isomorphic to the variety given by
\begin{equation*}
  t^2x + t(\zeta x^4)^2 + z^3 + xf(x,\zeta x^4,z) = 0
\end{equation*}
in $\PP(1,5,7)$. Note that the above quasihomogeneous polynomial is irreducible. Thus $M$ is irreducible. Since $M$ has the monomial $t^2x$, we have $\mult_{\msp}(M)\leq 2$. We write
\begin{equation*}
  D = b M + \Omega
\end{equation*}
where $b$ is nonnegative number and $\Omega$ is an effective divisor such that $M\not\subset \Supp(\Omega)$. By Corollary \ref{Foundation:bound of k-basis cor} we have $b\leq \frac{1}{5}$.
If $M$ is smooth at $\msp$ then we have
\begin{equation*}
  \frac{1}{\lambda} <\left (\Omega\cdot M\right)_{\msp} \leq(D - b M)\cdot M= \frac{6}{7} - \frac{12}{7}b .
\end{equation*}
This implies that $b<0$ and therefore, this is impossible. Thus $\mult_{\msp}(M) = 2$. Then we have
\begin{equation*}
  \left( \frac{1}{\lambda} - 2b \right)\cdot 2<\mult_{\msp}\left(D-bM\right) \mult_{\msp}(M)\leq(D - b M)\cdot M=\frac{6}{7} - \frac{12}{7}b.
\end{equation*}
It implies that $b > \frac{7}{8\lambda} - \frac{3}{8}$. This is impossible. Therefore the log pair $(S_{15},\lambda D)$ is log canonical along $S_{15} \setminus C_x$.
\end{proof}

\begin{lemma}
    The log pair $(S_{15}, \lambda D)$ is log canonical along $C_x \setminus \Sing(S_{15})$.
\end{lemma}

\begin{proof}
    Suppose that the log pair $(S_{15}, \lambda D)$ is not log canonical at some point $\msp\in S_{15}\setminus \Sing(S_{15})$. Since $\lambda a< 1$, using Lemma \ref{Inversion of adjunction for smooth point}, we get
\begin{equation}\label{inequality:H_x}
  \frac{1}{\lambda}<\mult_{\msp}(\Delta\cdot {C_x}) \leq \Delta\cdot C_x=(D - a C_x)\cdot C_x= \frac{3}{14} - \frac{3}{28}a.
\end{equation}
It implies that $a < 0$ which is impossible. Thus the log pair $(S_{15}, \lambda D)$ is log canonical along $\msp\in S_{15}\setminus \Sing(S_{15})$.
\end{proof}
We will now show that the log pair $(S_{15},\lambda D)$ is log canonical at the singular points $\msp_y$ and $\msp_t$.
\begin{lemma}
   The log pair $(S_{15},\lambda D)$ is log canonical at $\msp_y$.
\end{lemma}
\begin{proof}
    Suppose not, i.e. suppose that the log pair  $(S_{15},\lambda D)$ is not log canonical at $\msp_y$. Since $\msp_y \in C_x$ with $C_x$ smooth at $\msp_y$ and $\lambda a<1$, by Lemma \ref{Inversion of adjunction for singular point} we have
\begin{equation*}
  \frac{1}{4\lambda}<\mult_{\msp}(\Delta\cdot C_x)\leq  \Delta\cdot C_x = (D - a C_x)\cdot C_x =\frac{3}{14} - \frac{3}{28}a.
\end{equation*}
It implies that $a < 0$. It is impossible.
\end{proof}

\begin{lemma}
    The log pair $(S_{15}, \lambda D)$ is log canonical at $\msp_t$.
\end{lemma}

\begin{proof}
Let $\pi\colon \widetilde{S}_{15}\to S_{15}$ be the weighted blow-up at $\msp_t$ with weights $\mathrm{wt}(y)=3$ and $\mathrm{wt}(z)=2$. Then we have the following.
\begin{equation*}
  K_{\widetilde{S}_{15}} \equiv \pi^*(K_{S_{15}}) - \frac{2}{7}E,\qquad
  \widetilde{D} \equiv \pi^*(D)-mE,\qquad
  \widetilde{C}_x \equiv \pi^*(C_x) - \frac{6}{7}E,
\end{equation*}
where $E$ is the exceptional divisor and $m$ is any non-negative rational number. Using the above equivalences, we have
\begin{equation}\label{Lemma 14:inequality of m}
 0 \leq \widetilde{\Delta}\cdot \widetilde{C}_x= (\widetilde{D} - a\widetilde{C}_x)\cdot \widetilde{C}_x =\frac{3}{14} - m + \frac{3}{4}a.
\end{equation} From the above equations we have
\begin{equation*}
  K_{\widetilde{S}_{15}} + \lambda \widetilde{D} + \left(\lambda m + \frac{2}{7}\right)E \equiv \pi^*(K_{S_{15}} + \lambda D).
\end{equation*}
Thus the log pair $\left( \widetilde{S}_{15},  \lambda \widetilde{D} + \left(\lambda m + \frac{2}{7}\right)E\right)$ is not log canonical at some point $\msq\in E$.  

We will now bound $m$ by computing the volume of the pseudoeffective divisor $\pi^*(-K_{S_{15}})-uE$. 
Consider

\begin{equation*}
  \pi^*(-K_{S_{15}})-uE=\pi^*(2C_x) - uE = 2\widetilde{C}_x + \left( \frac{12}{7} - u \right)E.
\end{equation*}
Note that if $0\leq u \leq \frac{3}{14}$ then $\pi^*(2C_x) - uE$ is nef. Moreover if $\frac{3}{14}\leq u \leq \frac{12}{7}$ then the Zariski Decomposition is given by 
\begin{equation*}
  \pi^*(-K_{S_{15}})-uE=\pi^*(2C_x) - uE=\bigg(\left( \frac{16}{7} - \frac{4}{3}u \right)\widetilde{C}_x + \left( \frac{12}{7} - u \right)E\bigg)+\left(\frac{4x}{3}-\frac{2}{7}\right)\widetilde{C}_x.
\end{equation*}
We then have the following.
\begin{equation*}
    \mathrm{vol}(\pi^*(-K_{S_{15}})-uE)=
    \begin{cases}
        \dfrac{3}{7} - \dfrac{7}{6}u^2 & \textrm{for~} 0\leq u \leq \dfrac{3}{14},\\[10pt]
        \dfrac{u^2}{6}-\dfrac{4u}{7}+\dfrac{24}{49} & \textrm{for~} \dfrac{3}{14}\leq u \leq \dfrac{12}{7}.
    \end{cases}
\end{equation*}

Therefore, 
\begin{align*}
    m&\leq \frac{1}{(-K_{S_{15}})^2} \int_0^{\frac{12}{7}}\mathrm{vol}(\pi^*(-K_{S_{15}})-uE)\dd u+\epsilon_k\\
    &=\frac{7}{3}\left(\int_0^{\frac{3}{14}}\frac{3}{7} - \frac{7}{6}u^2 \dd u +\int_{\frac{3}{14}}^{\frac{12}{7}}\frac{u^2}{6}-\frac{4u}{7}+\frac{24}{49} \dd u\right)+\epsilon_k,\\
    &=\frac{9}{14}+\epsilon_k
\end{align*}
It implies that $ \lambda m + \frac{2}{7} < 1.$ Now we will consider different positions of the point $\msq$.

Suppose $\msq\in E \setminus\widetilde{C}_x$. Then the point $\msq$ is of type $\frac{1}{r}(\msa,\msb)$ for $r\leq 3$.
We have that the log pair
\begin{equation*}
  \left(\widetilde{S}_{15}, \lambda \widetilde{\Delta} + \left(\lambda m + \frac{2}{7}\right)E\right)
\end{equation*}
is not log canonical at $\msq$. Since $\lambda m+\frac{2}{7}<1$, using Lemma \ref{Inversion of adjunction for singular point}, we have
\begin{equation*}
 \frac{1}{3\lambda}< \mult_{\msq}\left(\widetilde{\Delta}\cdot E \right) \leq  (\widetilde{\Delta}\cdot E)=(\widetilde{D} - a\widetilde{C}_x)\cdot E=\frac{7}{6}m - a.
\end{equation*} Using \eqref{Lemma 14:inequality of m} and that $m\leq \frac{9}{14}+\epsilon_k$, we get a contradiction. Hence $\msq \in E\cap \widetilde{C}_x$. This implies that $\msq$ is a smooth point of $\widetilde{S}_{15}$. Then we have
\begin{equation*}
 \frac{1}{\lambda}< \mult_{\msq}(\widetilde{D}\cdot E)= \frac{7}{6}m \leq \frac{7}{6}\left(\frac{3}{14}+\frac{3a}{4}\right)=\frac{1}{4}+\frac{7a}{8}.
\end{equation*} Note that we have used the bound for $m$ from \eqref{Lemma 14:inequality of m}.

This implies that $a > \frac{8}{7\lambda} - \frac{2}{7}$, but $a\leq \frac{3}{4}$ by assumption. Thus, this gives a contradiction, hence proving our claim.
\end{proof}
This completes the proof of Proposition \ref{Proposition: 4}, which implies the following.
\begin{corollary}
  $\delta(S_{15}) \geq \frac{11}{10}$.
\end{corollary}

\subsection{$S_{16}$ in $\PP(1,4,5,8)$}
Let $S_{16}$ be a quasi-smooth hypersurface in $\PP(1,4,5,8)$ of degree $16$. By a suitable coordinate change we can assume that $S_{16}$ is given by a quasihomogeneous polynomial
\begin{equation*}
    t^2 - y^4 + xz^3 + x^2f_{14}(x,y,z) = 0
\end{equation*}
where $f_{14}(x,y,z)$ is a quasihomogeneous polynomial of degree $14$. The surface $S_{16}$ is singular at $\msp_z:=[0:0:1:0]$ of type $\frac{1}{5}(4,3)$ and at $Q_1\coloneqq [0:1:0:-1]$ and $Q_2\coloneqq [0:1:0:1]$, of type $\frac{1}{4}(1,1).$ Note that all these singular points lie on the curve $C_x$ which is isomorphic to the variety given by $(t + y^2)(t- y^2) = 0$ in $\PP(4,5,8)$. Thus we have
\begin{equation*}
    C_x = L_1 + L_2
\end{equation*}
where $L_1$ is given by $x = t + y^2 = 0$ and $L_2$ is given by $x = t - y^2 = 0$ in $\PP(4,5,8)$. In particular, the affine cone of $C_x$ is smooth at the points $Q_1$ and $Q_2$ and is singular at $\msp_z$. We have the following intersection numbers
\begin{equation*}
    L_1\cdot L_2 = \frac{2}{5},\qquad L_i^2 = -\frac{7}{20},\qquad C_x\cdot L_i = \frac{1}{20},\qquad C_x^2 = \frac{1}{10}.
\end{equation*}

Let $D$ be an anticanonical $\QQ$-divisor of $k$-basis type on $S_{16}$ with $k\gg 0$. We write
\begin{equation*}
    D = \alpha L_1 + \beta L_2 + \Delta
\end{equation*}
where $\alpha$ and $\beta$ are non-negative constants and $\Delta$ is an effective $\QQ$-divisor on the surface $S_{16}$ whose support does not contain the curves $L_1$ and $L_2$. We set $\lambda=\frac{40}{39}$. To determine that the log pair $(S_{16}, \lambda D)$ is log canonical we need some bounds and inequalities for $\alpha$ and $\beta$.

Since $L_2\not\subset \Supp(\Delta)$, we have
\begin{equation*}
    \frac{2}{5}\alpha = \alpha L_1\cdot L_2 \leq (D - \beta L_2)\cdot L_2 = \frac{1}{10} + \frac{7}{20}\beta. 
\end{equation*}
It implies that
\begin{equation}\label{bound on alpha and beta}
    \alpha \leq \frac{1}{4} + \frac{7}{8}\beta.
\end{equation}
This will come in handy for computations later. We will now bound the values of $\alpha$ and $\beta$.

\begin{lemma}
    $\alpha \leq \frac{17}{24}+\epsilon_k$ and $\beta \leq \frac{17}{24}+\epsilon_k$ where $\epsilon_k$ is a small constant depending on $k$ such that $\epsilon_k \to 0$ as $k\to \infty$.
\end{lemma}

\begin{proof}
Consider
\begin{equation*}
    2C_x - uL_2 = 2L_1 + (2 - u)L_2.
\end{equation*}
Since $L_2^2 < 0$, we have $2\geq u$. By multiplying the above equation by $L_1$, we have
\begin{equation*}
    (2C_x - uL_2)\cdot L_1 = \frac{1}{10} - \frac{2}{5}u.
\end{equation*}
Thus $\frac{1}{4}\geq u$ implies that $2C_x - uL_2$ is nef. For the case that $2\geq u\geq \frac{1}{4}$, we have
\begin{equation*}
    2C_x - uL_2 = (2-u)\left( \frac{8}{7}L_1 + L_2 \right) + \frac{8u-2}{7} L_1.
\end{equation*}
Thus
\begin{equation*}
    \vol(2C_x - uL_2) = 
    \begin{cases}
        -\dfrac{7}{20}u^2 - \dfrac{1}{5}u + \dfrac{2}{5} &\textrm{for~} 0\leq u \leq \dfrac{1}{4},\\[10pt]
        \dfrac{3}{28}(2-u)^2 &\textrm{for~}\dfrac{1}{4}\leq u \leq 2.
    \end{cases}
\end{equation*}

We have
\begin{align*}
    \beta &\leq \frac{1}{(-K_X)^2}\int_0^{2} \vol(2L_1+(2-u)L_2) \dd u +\epsilon_k \\
    &= \frac{5}{2}\bigg( \int_0^{\frac{1}{4}}-\frac{7}{20}u^2 - \frac{1}{5}u + \frac{2}{5} \dd u+\int_{\frac{1}{4}}^{2} \frac{3}{28}(2-u)^2\dd u\bigg)+\epsilon_k,\\
    &=\frac{5}{2}\bigg(\frac{353}{3840} + \frac{49}{256}\bigg)+\epsilon_k= \frac{17}{24}+\epsilon_k.
\end{align*}
Similar computations show that $\alpha \leq \frac{17}{24}+\epsilon_k.$
\end{proof}
This implies that $\lambda \alpha < 1$ and $\lambda \beta < 1$.

\begin{proposition}\label{Proposition: 5}
    The log pair $(S_{16},\lambda D)$  is log canonical. 
\end{proposition}

Let us now prove Proposition \ref{Proposition: 5}.

\begin{lemma}
    The log pair $(S_{16}, \lambda D)$ is log canonical along $S_{16}\setminus C_x$.
\end{lemma}

\begin{proof}
Suppose the log pair $(S_{16}, \lambda D)$ is not log canonical at a point $\msp\in S_{16}\setminus C_x$. This implies that $\msp$ is a smooth point of the surface $S_{16}$.
Note that $C_y$ is isomorphic to the variety given by 
\begin{equation}\label{polynomial of H_y}
    t^2 + xz^3 + x^2f_{14}(x,0,z) = 0
\end{equation}
in $\PP(1,5,8)$. Since the quasihomogeneous polynomial (\ref{polynomial of H_y}) has the monomial term $xz^3$, it is irreducible. It implies that $C_y$ is irreducible. Write
\begin{equation*}
  D = bC_y + \Omega
\end{equation*}
where $b$ is a nonnegative real number and $\Omega$ is an effective $\QQ$-divisor such that $C_y\not\subset \Supp(\Omega)$. From Corollary \ref{Foundation:bound of k-basis cor}, we see that $b \leq \frac{1}{5}$ and this implies that $\lambda b < 1$.

We know that $\mult_{\msp}(C_y)\leq 2$. Suppose that $C_y$ is smooth at $\msp$, that is $\mult_{\msp}(C_y)=1$. Then we have
\begin{equation*}
    \frac{1}{\lambda}<\Omega\cdot C_y=(D - bC_y)\cdot C_y = \frac{4}{5} - \frac{8}{5}b.
\end{equation*}
This implies that $b<0$ and this is a contradiction. 
 Thus we have $\mult_{\msp}(C_y)=2$. Since $\msp$ is a smooth point of $S_{16}$, we have that 
 $$\frac{1}{\lambda}<\mult_{\msp} (D) =\mult_{\msp}(bC_y+\Omega)=2b+\mult_{\msp}(\Omega).$$
 
 Using this we have, 
\begin{equation*}
    \left(\frac{1}{\lambda} - 2b\right)2 < \mult_{\msp}(\Omega) \mult_{\msp} (C_y) \leq \Omega\cdot C_y = \frac{4}{5} - \frac{8}{5}b.
\end{equation*}

This implies that
\begin{equation*}
    \left(\frac{2}{\lambda} - \frac{4}{5}\right) < b \leq  \frac{1}{5}.
\end{equation*}
From this we have $\frac{25}{16}<\lambda=\frac{5}{4}$, which is absurd.
\end{proof}

\begin{lemma}
    The log pair $(S_{16}, \lambda D)$ is log canonical along $C_x\setminus \Sing(S_{16})$.
\end{lemma}

\begin{proof}
We consider the case that $\msp\in C_x\setminus \Sing(S_{16})$. Without loss of generality we can assume that $\msp\in L_1$. Then the log pair
\begin{equation*}
  (S_{16},\lambda \alpha L_1+\lambda \Delta)
\end{equation*}
is not log canonical at $\msp$. Since $\lambda \alpha<1$, using Lemma \ref{Inversion of adjunction for smooth point} we get that 
\begin{equation*}
    \frac{1}{\lambda}<\Delta\cdot L_1=(D-\alpha L_1-\beta L_2)\cdot L_1 <\frac{1}{10}+\frac{7\alpha}{20}
\end{equation*}
Thus we have $\frac{20}{7}(\frac{1}{\lambda} - \frac{1}{10})<\alpha$. This is a contradiction.
\end{proof}

Now we will prove that the log pair $(S_{16},\lambda D)$ is log canonical at the singular points of $S_{16}$. 

\begin{lemma}
    The log pair $(S_{16}, \lambda D)$ is log canonical at $Q_1$ and $Q_2$.
\end{lemma}

\begin{proof}

Observe that $Q_i$ lies on $L_i$. We first consider $Q_1\in L_1$. So we have that the log pair $(S_{16},\lambda \alpha L_1+\lambda \Delta)$ is not log canonical at $Q_1$.

Since $\alpha L_1<1$, using Lemma \ref{Inversion of adjunction for singular point} and the inequality \eqref{bound on alpha and beta}, we get that 
\begin{align*}
    \frac{1}{4\lambda}<(\Delta\cdot L_1)_{Q_1}\leq \Delta\cdot L_1 &=(D-\alpha L_1-\beta L_2)\cdot L_1\\
    &=\left(\frac{1}{10} + \frac{7}{20}\alpha - \frac{2}{5}\beta\right)\\
    &\leq \left(\frac{1}{10} + \frac{7}{20}\bigg(\frac{1}{4}+\frac{7\beta}{20}\bigg)- \frac{2}{5}\beta\right)\\
    &=\frac{3}{16}\bigg(1-\frac{\beta}{2}\bigg).
\end{align*}
This implies that $\beta<0$. Similarly, for $Q_2\in L_2$ we can obtain $\alpha < 0$. They are absurd. 
\end{proof}

\begin{lemma}
    The log pair $(S_{16}, \lambda D)$ is log canonical at the point $\msp_z$ .
\end{lemma}

\begin{proof}
Suppose that the log pair $(S_{16}, \lambda D)$ is not log canonical at the point $\msp_z$. Since $\msp_z \in L_1 \cap L_2$, $L_1$, $L_2$ are smooth at $\msp_z$ and $\lambda \alpha<1$, $\lambda \beta<1$, using Lemma \ref{Inversion of adjunction for singular point}, we have 
\begin{align}\label{lemma 18:bounds on alpha and beta}
    \frac{1}{5\lambda}<\left(D-\alpha L_1\right)\cdot L_1&=\left (\beta L_2+\Delta\right)\cdot L_1=\frac{7\alpha}{20}+\frac{1}{10},\\
    \frac{1}{5\lambda}<\left(D-\beta L_2\right)\cdot L_2&=\left (\alpha L_1+\Delta\right)\cdot L_2=\frac{7\beta}{20}+\frac{1}{10}.
\end{align} These inequalities will come in handy in the computations that follow.

Let $\pi \colon \widetilde{S}_{16}\to S_{16}$ be the weighted blow-up at $\msp_z$ with weights $\mathrm{wt}(y)=4$ and $\mathrm{wt}(t)=3$. Then we have
\begin{equation*}
    \begin{array}{ll}
        K_{\widetilde{S}_{16}}\equiv \pi^*(K_{S_{16}}) + \dfrac{2}{5}E,&
        \widetilde{D}\equiv \pi^*(D)-mE,\\[10pt]
        \widetilde{L}_1\equiv \pi^*(L_1)-\dfrac{3}{5}E,&
        \widetilde{L}_2\equiv \pi^*(L_2)-\dfrac{3}{5}E,
    \end{array}
\end{equation*}
where $E$ is the exceptional divisor of $\pi$. The log pair
\begin{equation*}
  \left(\widetilde{S}_{16}, \lambda \widetilde{D} + \left(\lambda m - \frac{2}{5}\right)E \right)
\end{equation*}
is not log canonical at some point $\msq\in E$.

In order to bound the value of $m$, we will compute the volume of the pseudoeffective divisor $\pi^*(-K_{S_{16}})-uE$. This is given by 
\begin{equation*}
  \pi^*(2C_x)-uE\equiv 2\widetilde{L}_1 + 2\widetilde{L}_2 + \left(\frac{12}{5}-u\right)E.
\end{equation*}
Since $\widetilde{L}_1$ and $\widetilde{L}_2$ are negative definite $\tau(E) = \frac{12}{5}$. This divisor is nef when $u \in [0,\frac{2}{5}]$ and the Zariski Decomposition of the divisor when $u \in [\frac{2}{5},\frac{12}{5}]$ is given by 
$$
    2\widetilde{L}_1 + 2\widetilde{L}_2 + \left(\frac{12}{5}-u\right)E=  \left(\widetilde{L}_1 + \widetilde{L}_2 + E\right)\left(\frac{12}{5} - u\right)+\left(u-\frac{2}{5}\right)\left(\widetilde{L}_1+\widetilde{L}_2\right).
$$

Using this we then have
\begin{equation*}
  \vol(\pi^*(-K_{S_{16}}) - uE) =
    \begin{cases}
      (\pi^*(-K_{S_{16}})- uE)^2 = \dfrac{2}{5} - \dfrac{5}{12}u^2 & \textrm{for~} 0\leq u \leq \dfrac{2}{5},\\[10pt]
      \left(\widetilde{L}_1 + \widetilde{L}_2 + E\right)^2\left(\dfrac{12}{5} - u\right)^2 = \dfrac{1}{60}\left(\dfrac{12}{5} -u\right)^2 & \textrm{for~} \dfrac{2}{5}\leq u \leq \dfrac{12}{5}.
    \end{cases}
\end{equation*}
Then
\begin{equation*}
m \leq \frac{1}{\left(-K_{S_{16}}\right)^2}\int_0^{\tau(E)}\vol(-K_{S_{16}} - uE) \dd u + \epsilon_k = \frac{14}{15} + \epsilon_k.
\end{equation*}
It implies that $\lambda m-\frac{2}{5} < 1$. The point $\msq \in E$ is of type $\frac{1}{r}(\msa,\msb)$ where $r\leq 4$. Since $\widetilde{L_i}\cdot E = \frac{1}{4}$, $\widetilde{L}_i\cap E$ is the point with type $\frac{1}{4}(1,1)$.

Suppose $\msq \in E \setminus (\widetilde{L}_1\cup \widetilde{L}_2)$. Then $\msq$ is either a smooth point or a singular point of type $\frac{1}{3}(1,1)$.

Since we have that the log pair $(\widetilde{S}_{16}, \lambda \widetilde{\Delta}+(\lambda m-\frac{2}{5})E)$ is not log canonical at $\msq$ and $\left(\lambda m-\frac{2}{5}\right)<1$, using Lemma \ref{Inversion of adjunction for singular point}, we get

\begin{align*}
    \frac{1}{3\lambda}<\frac{1}{r\lambda}<\widetilde{\Delta}\cdot E=\frac{5m}{12}-\frac{\alpha}{4}-\frac{\beta}{4}.
\end{align*} From \eqref{lemma 18:bounds on alpha and beta}, this gives a contradiction, thus showing that $\msq \in \widetilde{L}_1\cap \widetilde{L}_2 \cap E$. Note that  $E\cap\widetilde{L}_1 = E\cap\widetilde{L}_2$ and hence this is the only remaining position of the point $\msq$.

Then we have that the log pair $(\widetilde{S}_{16}, \lambda \alpha \widetilde{L}_1+\beta \widetilde{L}_2+\lambda \widetilde{\Delta}+(\lambda m-\frac{2}{5})E)$ is not log canonical at $\msq$. We will consider a weighted blow-up $\pi_2\colon \widehat{S}_{16}\to S_{16}$ at $\msp_z$ with different weights, that is, let $\mathrm{wt}(y)=1$ and $\mathrm{wt}(z)=2$. From this we have
\begin{equation*}
    \begin{array}{ll}
      K_{\widehat{S}_{16}} \equiv \pi_2^*(K_{S_{16}}) - \dfrac{2}{5}F,&
     \widehat{D}\equiv \pi_2^*(D) - n F,\\[10pt]
      \widehat{L}_1\equiv \pi_2^*(L_1)-\dfrac{2}{5}F,&
      \widehat{L}_2\equiv \pi_2^*(L_2)-\dfrac{2}{5}F,
    \end{array}
\end{equation*}
where $F$ is the exceptional divisor of $\pi_2$.
The log pair
\begin{equation*}
  \left(\widehat{S}_{16}, \lambda\left(\alpha\widehat{L}_1 + \beta\widehat{L}_2 + \widehat{\Delta}\right) + \left(\lambda n + \frac{2}{5}\right)F \right)
\end{equation*}
is not log canonical at some point $\msq\in F$.

We will now bound $n$ using Corollary \ref{Foundation:bound of k-basis cor} for the divisor $\pi_2^*(-K_{\widehat{S}_{16}})-uF$.

\begin{equation*}
\pi_2^*(-K_{\widehat{S}_{16}})-uF\equiv 2\widehat{L}_1 + 2\widehat{L}_2 + \left(\frac{8}{5}-u\right)F.
\end{equation*}
Since $\widehat{L}_1$ and $\widehat{L}_2$ are negative definite we have $\tau(F) = \frac{8}{5}$.
This divisor is nef when $u\in [0,\frac{1}{10}]$. Therefore, the Zariski decomposition of the divisor when $u\in [\frac{1}{10},\frac{8}{5}]$ is given by 
\[
\pi_2^*(-K_{\widehat{S}_{16}})-uF=\left(\frac{32}{15}-\frac{4u}{3}\right)\left(\widehat{L}_1+\widehat{L}_2\right)+\left(\frac{8}{5}-u\right)F+\left(\frac{4u}{3}-\frac{2}{15}\right)\left(\widehat{L}_1+\widehat{L}_2\right).
\]
We also have
\begin{equation*}
  \vol(-K_{S_{16}} - uF) = 
    \begin{cases}
      (-K_{S_{16}} - uF)^2 = \dfrac{2}{5} - \dfrac{5}{2}u^2 & \textrm{for~} 0\leq u \leq \dfrac{1}{10}\\[10pt]
      \left(\dfrac{4}{3}\widehat{L}_1 + \dfrac{4}{3}\widehat{L}_2 + F\right)^2\left(\dfrac{8}{5} - u\right)^2 = \dfrac{1}{30}\left(\dfrac{8}{5} - u\right)^2 & \textrm{for~} \dfrac{1}{10}\leq u \leq \dfrac{8}{5}.
    \end{cases}
\end{equation*}
Then
\begin{equation*}
  n\leq \frac{1}{\left(-K_{S_{16}}\right)^2}\int_0^{\tau(F)}\vol(-K_{S_{16}} - uF) \dd u + \epsilon_k = \frac{17}{30} + \epsilon_k.
\end{equation*}
It implies that $\lambda n +\frac{2}{5}< 1$.

If $\msq \in F \setminus \left(\widehat{L}_1\cup \widehat{L}_2\right)$, then we have that the log pair $\left(\widehat{S}_{16}, \lambda \widehat{\Delta} + \left(\lambda n + \frac{2}{5}\right)F \right)$ is not log canonical at the point $\msq$. Using Lemma \ref{Inversion of adjunction for singular point}, we get 
\[
\frac{1}{2\lambda}<\widehat{\Delta}\cdot F=\frac{5n}{2}-\alpha-\beta
\] This gives a contradiction.

Therefore, $F \in \widehat{L}_1\cup \widehat{L}_2$. Without loss of generality, let us assume that $\msq \in F\cup \widehat{L}_1$. Then the log pair $\left(\widehat{S}_{16}, \lambda\left(\alpha\widehat{L}_1+ \widehat{\Delta}\right) + \left(\lambda n + \frac{2}{5}\right)F \right)$ is not log canonical at $\msq$. Using Lemma \ref{Inversion of adjunction for singular point}, we get 
\[
\frac{1}{\lambda}<\left(\alpha\widehat{L}_1+ \widehat{\Delta}\right)\cdot F=\frac{5n}{2}-\beta
\] Again, this gives a contradiction thus proving our claim.
\end{proof}
This completes the proof of Proposition \ref{Proposition: 5}, which implies the following.
\begin{corollary}
  $\delta(S_{16})\geq \frac{40}{39}$.
\end{corollary}

\subsection{$S_{18}$ in $\PP(1,4,6,9)$}

Let $S_{18}$ be a quasi-smooth hypersurface in $\PP(1,4,6,9)$ of degree $18$. By a suitable coordinate change we can assume that $S_{18}$ is given by a quasihomogeneous polynomial

\begin{equation*}
    t^2 + z^3 + y^3z + xf(x,y,z) = 0
\end{equation*}
where $f(x,y,z)$ is a quasihomogeneous polynomial of degree $17$. The singular points of $S_{18}$ are $\msp_y$, $[0:-1:1:0]$ and $[0:0:-1:1]$, of types $\frac{1}{4}(1,1)$, $\frac{1}{2}(1,1)$ and $\frac{1}{3}(1,1)$ respectively. Note that all the singular points of $S_{18}$ lie on $C_x$. Since $C_x$ is isomorphic to the variety given by the quasihomogeneous polynomial 
\begin{equation*}
    t^2 + z^3 + y^3z = 0
\end{equation*}
in $\PP(4,6,9)$ we can check that $C_x$ is quasi-smooth.  

Let $D$ be an anticanonical $\QQ$-divisor of $k$-basis type on $S_{18}$ with $k\gg 0$. We write
\begin{equation*}
    D = a C_x + \Delta
\end{equation*}
where $a$ is non-negative rational number and $\Delta$ is an effective $\QQ$-divisor such that $C_x\not\subset \Supp(\Delta)$. We set $\lambda = \frac{4}{3}$. 

\begin{proposition}\label{Proposition 6}
    The log pair  $(S_{18}, \lambda D)$ is log canonical.
\end{proposition}

Let us now prove Proposition \ref{Proposition 6}.

\begin{lemma}
    The log pair $(S_{18}, \lambda D)$ is log canonical along $S_{18}\setminus C_x$.
\end{lemma}

\begin{proof}
    Suppose that the log pair $(S_{18}, \lambda D)$ is not log canonical at some point $\msp\in S_{18}\setminus C_x$. Then $S_{18}$ is smooth at $\msp$. By a suitable coordinate change we can assume that $\msp = \msp_x$. Note that $C_y$ is isomorphic to the variety given by the equation
    \begin{equation*}
        t^2 + a_1x^9t + z^3 + a_2x^6z^2 + a_3x^{12}z = 0
    \end{equation*}
    in $\PP(1, 6, 9)$ where $a_1$, $a_2$ and $a_3$ are constants. It is irreducible and $\mult_{\msp}(C_y)\leq 2$. Write
    \begin{equation*}
        D = b C_y + \Omega
    \end{equation*}
    where $b$ is non-negative constant and $\Omega$ is an effective $\QQ$-divisor such that $C_y\not\subset \Supp(\Omega)$. From Corollary \ref{Foundation:bound of k-basis cor}, we have that $b \leq \frac{1}{5}$.
    
    If $\mult_{\msp}(C_y) = 1$ then, by Lemma \ref{Inversion of adjunction for smooth point} we have
    \begin{equation*}
        1 < \lambda\Omega\cdot C_y = \lambda \left(D-bC_y\right)\cdot C_y= \frac{8}{9} - \frac{16}{9}b. 
    \end{equation*} 
    This implies that $b<0$.
    
    Thus $\mult_{\msp}(C_y) = 2$. Then we have
    \begin{equation*}
        \frac{8}{3} - \frac{16}{3}b < \mult_{\msp}(\lambda\Omega)\mult_{\msp}(C_y) \leq \lambda \Omega\cdot C_y = \frac{8}{9} - \frac{16}{9}b.
    \end{equation*}
    It implies that $\frac{1}{2} < b$. By Corollary \ref{Foundation:bound of k-basis cor}, it is impossible. Therefore the log pair $(S_{18}, \lambda D)$ is log canonical along $S_{18}\setminus C_x$.
\end{proof}

\begin{lemma}
     The log pair $(S_{18}, \lambda D)$ is log canonical along $C_x$.
\end{lemma}

\begin{proof}
    Suppose that the log pair $(S_{18}, \lambda D)$ is not log canonical at some point $\msp\in C_x$. Then the singular point $\msp$ is of type $\frac{1}{r}(\msa,\msb)$ where $r\leq 4$. Since $\lambda a \leq 1$, by Lemma \ref{Inversion of adjunction for smooth point}, we have 
    \begin{equation*}
        \frac{1}{4\lambda} \leq \frac{1}{r\lambda} < \Delta\cdot C_x =(D - a C_x)\cdot C_x = \frac{1}{6} - \frac{a}{12}.
    \end{equation*}
  This implies that $a<0$ which is absurd. Thus the log pair $(S, \lambda D)$ is log canonical along $C_x$. 
\end{proof}

This completes the proof of Proposition \ref{Proposition 6}, which implies that the following.

\begin{corollary}
  $\delta(S_{18}) \geq \frac{4}{3}$.
\end{corollary}

\subsection{$S_{22}$ in $\PP(1,5,7,11)$} 

Let $S_{22}$ be a quasi-smooth hypersurface in $\PP(1,5,7,11)$ of degree $22$. By a suitable coordinate change we can assume that $S_{22}$ is given by a quasihomogeneous polynomial
\begin{equation*}
    t^2 + y^3z + xf(x,y,z) = 0
\end{equation*}
where $f(x,y,z)$ is a quasihomogeneous polynomial of degree $21$. The surface $S_{22}$ is singular at the points $\msp_y$ and $\msp_z$, of type $\frac{1}{5}(1,1)$ and $\frac{1}{7}(5,4)$, respectively. Note that the singular points of $S_{21}$ lie on $C_x$. Since $C_x$ is isomorphic to the variety given by the quasihomogeneous polynomial 
\begin{equation*}
    t^2 + y^3z = 0
\end{equation*}
in $\PP(5,7,11)$ we can check that it is irreducible. And the affine cone of $C_x$ is singular at $\msp_z$ except the vertex.

To determine the K-stability of $S_{22}$ it is sufficient to consider the following $\QQ$-divisors.

Let $D$ be an effective $\QQ$-Cartier divisor of $S_{22}$ such that $D\sim_{\QQ} -K_{S_{22}}$. We write
\begin{equation*}
    D = a C_x + \Delta
\end{equation*}
where $a$ is non-negative rational number and $\Delta$ is an effective $\QQ$-divisor such that $C_x\not\subset \Supp(\Delta)$. We set $\lambda = \frac{18}{17}$.

\begin{proposition}\label{Prop7}
    If $a\leq \frac{353}{504}$ then the log pair $(S_{22}, \lambda D)$ is log canonical. 
\end{proposition}

Let us now prove Proposition \ref{Prop7}.
\begin{lemma}
    The log pair $(S_{22}, \lambda D)$ is log canonical along $S_{22}\setminus C_x$.
\end{lemma}

\begin{proof}
    Suppose that the log pair $(S_{22}, \lambda D)$ is not log canonical at some point $\msp\in S_{22}\setminus C_x$. The surface $S_{22}$ is smooth at $\msp$. By a suitable coordinate change we can assume that $\msp = \msp_x$. Let $\mcL$ be the linear system that is given by $\alpha x^2y + \beta z = 0$ with $[\alpha:\beta]\in \PP^1$. Since the base locus of $\mcL$ is the finite points set, there is an effective divisor $M\in\mcL$ such that $M\not\subset \Supp(D)$. We have the inequality
    \begin{equation*}
        1 < \lambda D\cdot M = \frac{72}{85}.
    \end{equation*}
    It is impossible. Therefore the log pair $(S_{22}, \lambda D)$ is log canonical along $S_{22}\setminus C_x$.
\end{proof}

\begin{lemma}
  The log pair $(S_{22}, \lambda D)$ is log canonical along $C_x\setminus \{\msp_z\}$.
\end{lemma}

\begin{proof}
    Suppose that the log pair $(S_{22}, \lambda D)$ is not log canonical at some point $\msp\in C_x\setminus \{\msp_z\}$. Then the point $\msp$ is of type $\frac{1}{r}(\msa,\msb)$ where $r\leq 5$. Since $\lambda a < 1$, by Lemma \ref{Inversion of adjunction for smooth point}, we have 
    \begin{equation*}
        \frac{1}{5} \leq \frac{1}{r} < \lambda\Delta\cdot C_x = \lambda(D - a C_x)\cdot C_x = \frac{72}{595} - \frac{36}{595}a.
    \end{equation*}
    It implies that $a<0$ which is impossible. Thus the log pair $(S, \lambda D)$ is log canonical along $C_x\setminus \{\msp_z\}$.
\end{proof}

\begin{lemma}
    The log pair $(S_{22}, \lambda D)$ is log canonical at $\msp_z$.
\end{lemma}

\begin{proof}
    Let $\pi\colon \widetilde{S}_{22}\to S_{22}$ be the weighted blow-up at $\msp_z$ with weights $\wt(y) = 2$ and $\wt(t) = 3$. Then we have
\begin{equation*}
    K_{\widetilde{S}_{22}} \equiv \pi^*(K_{S_{22}}) - \frac{2}{7}E,\qquad
    \widetilde{D} \equiv \pi^*(D)- mE,\qquad
    \widetilde{C}_x \equiv\pi^*(C_x)-\frac{6}{7}E
\end{equation*}
where $E$ is the exceptional divisor of $\pi$. From the above equations we obtain
\begin{equation*}
    K_{\widetilde{S}_{22}} + \lambda \bar{D} + \left( \lambda m + \frac{2}{7} \right)E\equiv \pi^*\left(K_{S_{22}} + \lambda  D\right).
\end{equation*}
Then the log pair $\left(\widetilde{S}_{22}, \lambda \widetilde{D} + \left( \lambda m + \frac{2}{7} \right)E\right)$ is not log canonical at some point $\msq\in E$.

We have the following intersection numbers:
\begin{equation*}
    E^2 = -\frac{7}{6},\qquad \widetilde{C}_x^2 = -\frac{4}{5},\qquad\widetilde{D}\cdot \widetilde{C}_x = \frac{4}{35} - m,\qquad \widetilde{D}\cdot E = \frac{7}{6}m.
\end{equation*}

Meanwhile, the inequality
\begin{equation*}
    0\leq \widetilde{\Delta}\cdot \widetilde{C}_x = (\widetilde{D} - a\widetilde{C}_x)\cdot \widetilde{C}_x = \frac{4}{35} - m + \frac{4}{5}a
\end{equation*}
implies that 
\begin{equation}\label{No. 9:bound of alpha}
    m \leq \frac{4}{35} + \frac{4}{5}a.
\end{equation}

If $\msq\in \widetilde{C}_x$ then, by Lemma \ref{Inversion of adjunction for smooth point} we have the following inequality.
\begin{equation*}
    1 < \bigg(\lambda\widetilde{\Delta}+\big(\lambda m+\frac{2}{7}\big)E\bigg)\cdot \widetilde{C}_x = \lambda\left(\frac{4}{35} + \frac{4}{5}a\right)+\frac{2}{7}\leq 1,
\end{equation*}
which is absurd. Thus $\msq\not\in \widetilde{C}_x$.

The point $\msq \in E \setminus \widetilde{C}_x$ is of type $\frac{1}{r}(\msa,\msb)$ with $r\leq 3$. Since $\lambda m + \frac{2}{7} \leq 1$, by Lemma \ref{Inversion of adjunction for singular point} and the inequality (\ref{No. 9:bound of alpha}) we have the following:
\begin{equation*}
    \frac{1}{3\lambda }\leq \frac{1}{r\lambda} < \widetilde{\Delta}\cdot E =  \frac{7}{6}m - a \leq \frac{2}{15} - \frac{a}{15}.
\end{equation*}
This implies that $a<0$, which is absurd. Thus the log pair $(S_{22}, \lambda D)$ is log canonical at $\msp_z$.
\end{proof}

This completes the proof of Proposition \ref{Prop7}, which implies the following. 

\begin{corollary}\label{Cor:estimation of delta by proposition}
$\delta(S_{22}) \geq \frac{18}{17}$.
\end{corollary}

\begin{proof}
    Suppose that $D$ is $k$-basis type on $S_{22}$ with $k\gg 0$. By Corollary \ref{Foundation:bound of k-basis cor} we have $a\leq \frac{2}{3} + \epsilon_k$ where $\epsilon_k$ is a small constant depending on $k$ such that $\epsilon_k \to 0$ as $k\to \infty$. By Proposition \ref{Prop7} the log pair $(S_{22}, \lambda D)$ is log canonical.
\end{proof}

\subsection{$S_{30}$ in $\PP(1,6,10,15)$} 

Let $S_{30}$ be a quasi-smooth hypersurface in $\PP(1,6,10,15)$ of degree $30$. By a suitable coordinate change we can assume that $S_{30}$ is given by a quasihomogeneous polynomial
\begin{equation*}
    t^2 + z^3 + y^5 + xf(x,y,z) = 0
\end{equation*}
where $f(x,y,z)$ is a quasihomogeneous polynomial of degree $29$. The surface $S_{30}$ is singular at the points $[0:-1:1:0]$, $[0:-1:0:1]$ and $[0:0:-1:1]$ of types $\frac{1}{2}(1,1)$, $\frac{1}{3}(1,1)$ and $\frac{1}{5}(1,1)$, respectively. Note that all the singular points of $S_{30}$ lie on the curve $C_x$. Since $C_{x}$ is isomorphic to the variety given by the quasihomogeneous polynomial 
\begin{equation*}
    t^2 + z^3 + y^5 = 0
\end{equation*}
in $\PP(6,10,15)$ we can check that $C_x$ is quasi-smooth.

To determine the K-stability of $S_{30}$ it is sufficient to consider the following $\QQ$-divisors.
Let $D$ be an effective $\QQ$-Cartier divisor of $S_{30}$ such that $D\sim_{\QQ} -K_{S_{30}}$. We write
\begin{equation*}
    D = aC_x + \Delta
\end{equation*}
where $a$ is non-negative rational number and $\Delta$ is an effective $\QQ$-divisor such that $C_x\not\subset \Supp(\Delta)$. We set $\lambda =\frac{4}{3}$.

\begin{proposition}\label{Proposition 8}
    If $a\leq \frac{3}{4}$ then the log pair $(S_{30},\lambda D)$ is log canonical. 
\end{proposition}
We will now prove Proposition \ref{Proposition 8}.
\begin{lemma}
    The log pair $(S_{30}, \lambda D)$ is log canonical along $S_{30}\setminus C_x$.
\end{lemma}

\begin{proof}
    Suppose that the log pair $(S_{30}, \lambda D)$ is not log canonical at some point $\msp\in S_{30}\setminus C_x$. The surface $S_{30}$ is smooth at $\msp$. By a suitable coordinate change we can assume that $\msp = \msp_x$. Let $\mcL$ be the linear system that is given by $\alpha x^4y + \beta z = 0$ with $[\alpha:\beta]\in \PP^1$. Since the base locus of $\mcL$ is the finite points set, there is an effective divisor $M\in\mcL$ such that $M\not\subset \Supp(D)$. We have the inequality
    \begin{equation*}
        1< \lambda D\cdot M = \frac{8}{9}. 
    \end{equation*}
    It is impossible. Therefore the log pair $(S_{30}, \lambda D)$ is log canonical along $S_{30}\setminus C_x$.
\end{proof}

\begin{lemma}
    The log pair $(S_{30}, \lambda D)$ is log canonical along $C_x$.
\end{lemma}

\begin{proof}
    Suppose that the log pair $(S_{30}, \lambda D)$ is not log canonical at some point $\msp\in C_x$. Then the  point $\msp$ is of type $\frac{1}{r}(\msa,\msb)$ where $r\leq 5$. Since $\lambda a \leq 1$, by Lemma \ref{Inversion of adjunction for smooth point}, we have the inequality
    \begin{equation*}
        \frac{1}{5\lambda} \leq \frac{1}{r\lambda} < \Delta\cdot C_x = (D - aC_x)\cdot C_x = \left(\frac{1}{15}-\frac{1}{30}a\right).
    \end{equation*}
    It implies that $a<0$. It is impossible. Thus the log pair $(S, \lambda D)$ is log canonical along $C_x$.
\end{proof}

This completes the proof of Proposition \ref{Proposition 8} which implies  the following.
\begin{corollary}
    $\delta(S_{30})\geq \frac{4}{3}$.
\end{corollary}

\begin{proof}
    The proof of this Corollary is similar to the proof of Corollary \ref{Cor:estimation of delta by proposition}. 
\end{proof}

\subsection{$S_{36}$ in $\PP(1,7,12,18)$}

Let $S_{36}$ be a quasi-smooth hypersurface in $\PP(1,7,12,18)$ of degree $36$. By a suitable coordinate change we can assume that $S_{36}$ is given by a quasihomogeneous polynomial
\begin{equation*}
    t^2 + z^3 + xf(x,y,z) = 0
\end{equation*}
where $f(x,y,z)$ is a quasihomogeneous polynomial of degree $35$. The surface $S_{36}$ is singular at the points $\msp_y$ and $[0:0:-1:1]$, of type $\frac{1}{7}(2,3)$ and $\frac{1}{6}(1,1)$ respectively. Note that all singular points of $S_{36}$ lie on the curve $C_x$. Since $C_x$  is isomorphic to the variety given by the quasihomogeneous polynomial 
\begin{equation*}
    t^2 + z^3 = 0
\end{equation*}
in $\PP(7,12,18)$ we can check that $C_x$ is irreducible. 

Let $D$ be an effective $\QQ$-Cartier divisor of $S_{36}$ such that $D\sim_{\QQ} -K_{S_{36}}$. We can write
\begin{equation*}
    D = a C_x + \Delta
\end{equation*}
where $a$ is non-negative rational number and $\Delta$ is an effective $\QQ$-divisor such that $C_x\not\subset \Supp(\Delta)$. We set $\lambda=\frac{8}{7}$

\begin{proposition}\label{Proposition:9}
    If $a\leq \frac{11}{16}$ then the log pair $(S_{36},\lambda D)$ is log canonical. 
\end{proposition} 

We will now prove Proposition \ref{Proposition:9}.

\begin{lemma}
    The log pair $(S_{36}, \lambda D)$ is log canonical along $S_{36}\setminus C_x$.
\end{lemma}

\begin{proof}
    Suppose that the log pair $(S_{36}, \lambda D)$ is not log canonical at some point $\msp\in S_{36}\setminus C_x$. By a suitable coordinate change we can assume that $\msp = \msp_x$. Let $\mcL$ be the linear system that is given by $\alpha x^5y + \beta z = 0$ with $[\alpha:\beta]\in \PP^1$. Since the base locus of $\mcL$ is the finite points set, there is an effective divisor $M\in\mcL$ such that $M\not\subset \Supp(D)$. Since $S_{36}$ is smooth at $\msp$ by Lemmas \ref{Foundation:mult} and \ref{Foundation:inequality mult intersection} we have the inequality
    \begin{equation*}
        1 < \lambda D\cdot M = \frac{32}{49}.
    \end{equation*}
    It is impossible. Therefore the log pair $(S_{36}, \lambda D)$ is log canonical along $S_{36}\setminus C_x$.
\end{proof}

\begin{lemma}
    The log pair $(S_{36}, \lambda D)$ is log canonical along $C_x\setminus \{\msp_y\}$.
\end{lemma}

\begin{proof}
    Suppose that the log pair $(S_{36}, \lambda D)$ is not log canonical at some point $\msp\in C_x\setminus \{\msp_y\}$. Then the point $\msp$ is of type $\frac{1}{r}(\msa,\msb)$ where $r\leq 6$. Since $\lambda a \leq 1$, by Lemma \ref{Inversion of adjunction for smooth point}, we have the inequality
    \begin{equation*}
        \frac{1}{6} \leq \frac{1}{r} < \lambda \Delta\cdot C_x = \lambda(D - a C_x)\cdot C_x = \lambda\left(\frac{1}{21}-\frac{1}{42}a \right).
    \end{equation*}
    It implies that $a <0$. It is impossible. Thus the log pair $(S, \lambda D)$ is log canonical along $C_x\setminus \{\msp_y\}$.
\end{proof}

\begin{lemma}
    The log pair $(S_{36}, \lambda D)$ is log canonical at $\msp_y$.
\end{lemma}

\begin{proof}
    Let $\pi\colon \widetilde{S}_{36}\to S_{36}$ be the weighted blow-up at $\msp_y$ with weights $\wt(z) = 2$ and $\wt(t) = 3$. Then we have
\begin{equation*}
    K_{\widetilde{S}_{36}} \equiv \pi^*(K_{S_{36}}) - \frac{2}{7}E,\qquad
   \widetilde{D} \equiv \pi^*(D)- mE,\qquad
   \widetilde{C}_x \equiv \pi^*(C_x)- \frac{6}{7}E 
\end{equation*}
where $E$ is the exceptional divisor of $\pi$. From the above equations we obtain
\begin{equation*}
    K_{\widetilde{S}_{36}} + \lambda\widetilde{D} + \left( \lambda m+ \frac{2}{7} \right)E \equiv \pi^*\left(K_{S_{36}} + \lambda D\right).
\end{equation*}
Then the log pair $\left(\widetilde{S}_{36}, \lambda\widetilde{D} + \left( \lambda m + \frac{2}{7} \right)E\right)$ is not log canonical at some point $\msq\in E$.

We have the following intersection numbers:
\begin{equation*}
    E^2 = -\frac{7}{6},\qquad \widetilde{C}_x^2 = -\frac{5}{6},\qquad\widetilde{D}\cdot \widetilde{C}_x = \frac{1}{21} -m,\qquad \widetilde{D}\cdot E = \frac{7}{6}m.
\end{equation*}

Meanwhile, the inequality
\begin{equation*}
    0 \leq (\widetilde{D} - a \widetilde{C}_x)\cdot \widetilde{C}_x = \frac{1}{21} - m + \frac{5}{6}a
\end{equation*}
implies that 
\begin{equation}\label{in of no. 12}
    m \leq \frac{1}{21} + \frac{5}{6}a.
\end{equation}

If $\msq\in \widetilde{C}_x$ then, by Lemma \ref{Inversion of adjunction for smooth point} we have the following inequality.
\begin{equation*}
    1 < \widetilde{C}_x\cdot \left(\lambda\widetilde{\Delta}+\big(\lambda m+\frac{2}{7}E\big)\right) = \lambda\left(\frac{1}{21}+ \frac{5}{6}a\right)+\frac{2}{7}\leq \frac{195}{196},
\end{equation*}
which is absurd. Thus $\msq\not\in \widetilde{C}_x$. It implies that the log pair $(\widetilde{S}_{36},\lambda \widetilde{\Delta}+(\lambda m+\frac{2}{7}E))$ is not log canonical at $\msq$. The point $\msq$ is a point of type $\frac{1}{r}(\msa,\msb)$ with $r\leq 3$. Since $\lambda m + \frac{2}{7} \leq 1$, by Lemma \ref{Inversion of adjunction for singular point} and the inequality (\ref{in of no. 12}) we have
\begin{equation*}
    \frac{1}{3\lambda}\leq \frac{1}{r\lambda} < \widetilde{\Delta}\cdot E = \big(\widetilde{D}-a\widetilde{C}_x\big)\cdot E=\frac{7}{6}m - a\leq \frac{1}{18} - \frac{a}{36},
\end{equation*}
which implies that $a<0$. Therefore the log pair $(S_{36}, \lambda D)$ is log canonical at $\msp_y$.
\end{proof}
This completes the proof of Proposition \ref{Proposition:9}, which implies that 
\begin{corollary}
    $\delta(S)\geq \frac{8}{7}$.
\end{corollary}

\begin{proof}
    The proof of this Corollary is similar to the proof of Corollary \ref{Cor:estimation of delta by proposition}. 
\end{proof}

\subsection{$S_{40}$ in $\PP(1,8,13,20)$}

Let $S_{40}$ be a quasi-smooth hypersurface of degree $40$ in $\PP(1,8,13,20)$ . By a suitable coordinate change we can assume that $S_{40}$ is given by a quasihomogeneous polynomial
\begin{equation*}
    t^2 + y^5 + xf(x,y,z) = 0
\end{equation*}
where $f(x,y,z)$ is a quasihomogeneous polynomial of degree $39$. The surface $S_{40}$ is singular at the points $\msp_z$ of type $\frac{1}{13}(2,5)$ and $Q:=[0:-1:0:1]$ of type $\frac{1}{4}(1,1)$. Note that the hyperplane section $C_x$ that is cut out by the equation $x=0$ in $S_{40}$, is isomorphic to the variety given by the quasihomogeneous polynomial 
\begin{equation*}
    t^2 + y^5 = 0
\end{equation*}
in $\PP(8,13,20)$, and contains both the singular points of the surface $S_{40}$. From the equation of $C_x$ we can also check that $C_x$ is irreducible. 

Let $D$ be an anticanonical $\QQ$-divisor of $k$-basis type on $S_{40}$ with $k\gg 0$. We write 
\begin{equation*}
    D = aC_x + \Delta
\end{equation*}
where $a$ is a non-zero constant and $\Delta$ is an effective $\QQ$-divisor such that $C_x\not\subset \Supp(\Delta)$. We set $\lambda=\frac{79}{78}$.

\begin{proposition}\label{proposition:No10}
The log pair $(S_{40}, \lambda D)$ is log canonical.
\end{proposition}

We will now prove Proposition \ref{proposition:No10}.

\begin{lemma}
    The log pair $(S_{40}, \lambda D)$ is log canonical along $S_{40}\setminus C_x$.
\end{lemma}

\begin{proof}
    Suppose that the log pair $(S_{40}, \lambda D)$ is not log canonical at some point $\msp\in S_{40}\setminus C_x$. By a suitable coordinate change we can assume that $\msp = \msp_x$. Let $\mcL$ be the linear system that is given by $\alpha x^5y + \beta z = 0$ with $[\alpha:\beta]\in \PP^1$. Since the base locus of $\mcL$ is the finite points set, there is an effective divisor $M\in\mcL$ such that $M\not\subset \Supp(D)$. Since $S_{40}$ is smooth at $\msp$ by Lemmas \ref{Foundation:mult} and \ref{Foundation:inequality mult intersection} we have the inequality
    \begin{equation*}
        1 < \lambda D\cdot M = \frac{79}{156}.
    \end{equation*}
    It is impossible. Therefore the log pair $(S_{40}, \lambda D)$ is log canonical along $S_{40}\setminus C_x$.
\end{proof}

\begin{lemma}
  The log pair $(S_{40}, \lambda D)$ is log canonical along $C_x\setminus \{\msp_z\}$.
\end{lemma}

\begin{proof}
    Suppose that the log pair $(S_{40}, \lambda D)$ is not log canonical at some point $\msp\in C_x\setminus \{\msp_z\}$. Then the point $\msp$ is either a smooth point of $S_{40}$ or the singular point $Q$. By Lemma \ref{Foundation:bound of k-basis cor} we have $\lambda a \leq 1$. By Lemma \ref{Inversion of adjunction for singular point}, we have the inequality
    \begin{equation*}
        \frac{1}{4} < \lambda\Delta\cdot C_x =\lambda \bigg(\frac{1}{26}-\frac{a}{52}\bigg).
    \end{equation*}
    It implies that $a<0$. It is impossible. Thus the log pair $(S, \lambda D)$ is log canonical along $C_x\setminus \{\msp_z\}$.
\end{proof}

\begin{lemma}
    The log pair $(S_{40}, \lambda D)$ is log canonical at $\msp_z$.
\end{lemma}

\begin{proof}
    Let $\pi\colon \widetilde{S}_{40}\to S_{40}$ be the weighted blow-up at $\msp_z$ with weights $\wt(y) = 2$ and $\wt(t) = 5$. Then
\begin{equation*}
     K_{\widetilde{S}_{40}} \equiv \pi^*(K_{S_{40}}) - \frac{6}{13}E,\qquad
     \widetilde{D}\equiv\pi^*(D) - m E,\qquad
     \widetilde{C}_x\equiv \pi^*(C_x) - \frac{10}{13}E.
\end{equation*}

where $m$ is a non-negative constant and $E$ is the exceptional divisor of $\pi$. From the above equations, we have
    \begin{equation*}
        K_{\widetilde{S}_{40}} + \lambda \widetilde{D} + \left( \lambda m + \frac{6}{13} \right)E \equiv \pi^*\left(K_{S_{40}} + \lambda D\right).
    \end{equation*}
Then the log pair $\left(\widetilde{S}_{40}, \lambda \widetilde{D} + \left( \lambda m + \frac{6}{13} \right)E\right)$ is not log canonical at some $\msq\in E$.

We have the following intersection numbers:
\begin{equation*}
    E^2 = -\frac{13}{10},\qquad \widetilde{C}_x^2 = -\frac{3}{4},\qquad\widetilde{D}\cdot \widetilde{C}_x = \frac{1}{26} - m,\qquad \widetilde{D}\cdot E = \frac{13}{10}m.
\end{equation*}

By Corollary \ref{Foundation:bound of k-basis cor} we have $a \leq \frac{3}{4}$ for $k\gg 0$. 
To find a bound of the constant $m$ we compute the volume of the divisor $\pi^*(-K_{S_{40}}) - uE$ where $u$ is a non-negative real number. Consider
\begin{equation*}
    \pi^*(-K_{S_{40}}) - uE\equiv 2\widetilde{C}_x + \left(\frac{20}{13} - u\right)E.
\end{equation*}
Since $\widetilde{C}_x^2 < 0$, we have $\tau(E) =  \frac{20}{13}$, and $\vol(\pi^*(-K_{S_{40}}) - uE) = 0$ for $u> \tau(E)$. Since
\begin{equation*}
    \left(2\widetilde{C}_x + \left(\frac{20}{13} - u\right)E\right)\cdot \widetilde{C}_x = \frac{1}{26} - u,
\end{equation*}
$\pi^*(-K_{S_{40}}) - uE$ is nef for $0\leq u \leq \frac{1}{26}$. Thus we have
\begin{equation*}
    \vol(\pi^*(-K_{S_{40}}) - uE) = (\pi^*(-K_{S_{40}}) - uE)^2 = \frac{1}{13} - \frac{13}{10}u^2
\end{equation*}
for $0\leq u \leq \frac{1}{26}$. We consider the case when $\frac{1}{26}\leq u \leq \tau(E)$. Since 
\begin{equation*}
    \left(\frac{20}{13} - u\right)\left(\frac{4}{3}\widetilde{C}_x + E\right)\cdot \widetilde{C}_x = 0,
\end{equation*}
the Zariski decomposition of $\pi^*(-K_{S_{40}}) - uE$, in this interval, is 
\begin{equation*}
    \left(\frac{20}{13} - u\right)\left(\frac{4}{3}\widetilde{C}_x + E\right) + \left(-\frac{2}{39} + \frac{4}{3}u\right)\widetilde{C}_x.
\end{equation*}
From this we have
\begin{equation*}
    \vol(\pi^*(-K_{S_{40}}) - uE) = \left(\frac{20}{13} - u\right)^2\left(\frac{4}{3}\widetilde{C}_x + E\right)^2 = \left(\frac{20}{13} - u\right)^2\frac{1}{30}.
\end{equation*}
Consequently, we have
\begin{multline*}
    m \leq \frac{1}{(-K_{S_{40}})^2}\int_{0}^{\frac{20}{13}} \vol(\pi^*(-K_{S_{40}}) - uE) \dd u + \epsilon_k \\ = 13\left(\int_{0}^{\frac{1}{26}} \frac{1}{13} - \frac{13}{10}u^2 \dd u + \int_{\frac{1}{26}}^{\tau(E)} \left(\frac{20}{13} - u\right)^2\frac{1}{30} \dd u \right) + \epsilon_k = \frac{41}{78} + \epsilon_k
\end{multline*}
where $\epsilon_k$ is a small constant depending on $k$ such that $\epsilon_k\to 0$ as $k\to \infty$. For $k\gg 0$ we can assume that 
\begin{equation*}
    m\leq \frac{42}{79}.
\end{equation*}
It implies that $\lambda m + \frac{6}{13} \leq 1$.

Meanwhile, the inequality
\begin{equation*}
    0 \leq (\widetilde{D} - a \widetilde{C}_x)\cdot \widetilde{C}_x = \frac{1}{26} - m + \frac{3}{4}a
\end{equation*}
implies that 
\begin{equation*}
    m \leq \frac{1}{26} + \frac{3}{4}a.
\end{equation*}

If $\msq\in \widetilde{C}_x$ then, by Lemma \ref{Inversion of adjunction for smooth point} we have
\begin{equation*}
    \frac{1}{\lambda} < \widetilde{D}\cdot E = \frac{13}{10}m\leq \frac{13}{10}\left(\frac{1}{26} + \frac{3}{4}a\right).
\end{equation*}
It is impossible. Thus $\msq\not\in \widetilde{C}_x$. It implies that the log pair $(\widetilde{S}_{40}, \lambda\widetilde{\Delta} + ( \lambda m + \frac{6}{13} )E)$ is not log canonical at the point $\msq\in E\setminus \widetilde{C}_x$.

Since the point $\msq$ is of type $\frac{1}{r}(a,b)$ with $r\leq 5$, by Lemma \ref{Inversion of adjunction for singular point} and the above inequality, we have 
\begin{equation*}
    \frac{1}{5\lambda}\leq \frac{1}{r \lambda} < \widetilde{\Delta}\cdot E = \frac{13}{10}m- a\leq \frac{1}{20} - \frac{a}{40}.
\end{equation*} 
It is impossible. Therefore the log pair $(S_{40}, \lambda D)$ is log canonical at $\msp_z$.
\end{proof}

This completes the proof of Proposition \ref{proposition:No10}, which also implies the following.
\begin{corollary}
  $\delta(S_{40}) \geq \frac{79}{78}$.
\end{corollary}

\begin{Ack}
    The authors would like to thank Prof. Ivan Cheltsov for his comments and useful discussions. The first author was supported by the National Research Foundation of Korea (NRF) grant funded by the Korea government (MSIP) (NRF-2020R1A2C4002510). The second author was supported by the EPSRC New Horizons Grant No.EP/V048619/1. The third author was supported by the National Research Foundation of Korea (NRF-2020R1A2C1A01008018).
\end{Ack}

\end{document}